\documentclass[12pt,reqno]{amsart}
\usepackage{amsfonts}
\usepackage{amssymb}
\numberwithin{equation}{section}
\theoremstyle{plain}
 \newtheorem{thm}{Theorem}[section]

 \newtheorem{prop}[thm]{Proposition}
\theoremstyle{definition}
 \newtheorem{defn}[thm]{Definition}
 \newtheorem{ex}[thm]{Example}
 
\newcommand{\al}{\alpha}
\newcommand{\bt}{\beta}
\newcommand{\gm}{\gamma}
\newcommand{\Gm}{\Gamma}
\newcommand{\dl}{\delta}

\newcommand{\ep}{\varepsilon}

\newcommand{\ld}{\lambda}
\newcommand{\Ld}{\Lambda}
\newcommand{\sg}{\sigma}

\newcommand{\ph}{\varphi}
\newcommand{\Ph}{\Phi}
\newcommand{\ps}{\psi}
\newcommand{\Ps}{\Psi}
\newcommand{\rh}{\rho}
\newcommand{\Up}{\Upsilon}

\newcommand{\dar}{\downarrow}
\newcommand{\uar}{\uparrow}

\newcommand{\mcal}{\mathcal}
\newcommand{\mrm}{\mathrm}
\newcommand{\mfr}{\mathfrak}

\newcommand{\wh}{\widehat}
\newcommand{\wt}{\widetilde}
\newcommand{\la}{\langle}
\newcommand{\ra}{\rangle}
\newcommand{\ts}{\textstyle}

\newcommand{\R}{\mathbb{R}}

\newcommand{\les}{\leqslant}
\newcommand{\ges}{\geqslant}
\newcommand{\law}{\mathcal L}
\newcommand{\cl}{\colon}

\newcommand{\Rpl}{\mathbb{R}_{+}^{\circ}}
\renewcommand{\labelenumi}{(\roman{enumi})}

\begin{document}
\setlength{\baselineskip}{16pt}
\setlength{\parindent}{1.8pc}
\allowdisplaybreaks
\hyphenation{self-de-com-pos-able self-decom-pos-abil-ity de-com-pos-abil-ity}

{\bf
{\large
\noindent
Inversions of  infinitely divisible distributions and\\
conjugates of stochastic integral mappings}

\vspace{5mm}
\noindent
Ken-iti Sato\footnote{Hachiman-yama 1101-5-103, Tenpaku-ku, Nagoya, 
468-0074 Japan.\\
E-mail: ken-iti.sato@nifty.ne.jp}} 

\vspace{5mm}
\begin{quotation}
{\small
The dual of an  infinitely divisible distribution on $\R^d$ without 
Gaussian part defined in Sato, ALEA {\bf 3} (2007), 67--110, is renamed to
the inversion.  Properties and characterization
of the inversion are given.  
A stochastic integral mapping is a mapping $\mu=\Phi_f\,\rho$ of
$\rho$ to $\mu$ in the class of infinitely divisible distributions
on $\mathbb R^d$, where $\mu$ is the distribution of an improper
stochastic integral of a nonrandom function $f$ with respect to a 
L\'evy process on $\mathbb R^d$ with distribution $\rho$ at time $1$.
The concept of the conjugate is
introduced for a class of stochastic integral mappings and its close connection with
the inversion is shown.  The domains and ranges of the conjugates  
of three two-parameter families of stochastic integral mappings
are described.  Applications to the study of the limits 
of the ranges of iterations of stochastic integral mappings are made.

}
\end{quotation}

\section{Introduction}

Let $ID=ID(\R^d)$ be the class of  infinitely divisible distributions on 
the $d$-dimensional Euclidean space $\R^d$.  
We use the L\'evy--Khintchine representation of the
characteristic function $\wh\mu(z)$ of $\mu\in ID$ in the form
\[
\wh\mu(z)=\exp\Bigl[ -\tfrac12 \langle z,A_{\mu} z\rangle +\int_{\mathbb{R}^d}
(e^{i\langle z,x\rangle}-1-i\langle z,x\rangle
1_{\{|x|\les 1\}}(x) ) \nu_{\mu}(dx)
+i\langle \gamma_{\mu},z\rangle\Bigr],\quad z\in\R^d,
\]
where $A_{\mu}$, $\nu_{\mu}$, and $\gm_{\mu}$ are the Gaussian covariance
matrix, the L\'evy measure, and the location parameter of $\mu$, respectively.
A measure $\nu$ on $\R^d$ is the L\'evy measure of some $\mu\in ID$ if and only 
if $\nu(\{ 0\})=0$ and $\int_{\R^d} (|x|^2\land 1)\nu(dx)<\infty$.  
The class of the triplets $(A_{\mu},\nu_{\mu},\gm_{\mu})$ represents the 
class $ID$ one-to-one.
If $\int_{|x|\les1}
|x|\nu_{\mu}(dx)<\infty$, then $\mu$ is said to have drift and $\wh\mu(z)$
has expression
\[
\wh\mu(z)=\exp\Bigl[-\tfrac12 \langle z,A_{\mu} z\rangle +\int_{\mathbb{R}^d}
(e^{i\langle z,x\rangle}-1) \nu_{\mu}(dx)+i\langle \gamma_{\mu}^0,z\rangle
\Bigr],
\]
where $\gm_{\mu}^0$ is called the drift of $\mu$.  If $\int_{|x|>1}
|x|\nu_{\mu}(dx)<\infty$, then $\mu$ has mean $m_{\mu}$ and
\[
\wh\mu(z)=\exp\Bigl[-\tfrac12 \langle z,A_{\mu} z\rangle +\int_{\mathbb{R}^d}
(e^{i\langle z,x\rangle}-1-i\langle z,x\rangle) \nu_{\mu}(dx)+i\langle 
m_{\mu},z\rangle\Bigr].
\]
Conversely, if $\mu\in ID$ has mean, then $\int_{|x|>1}
|x|\nu_{\mu}(dx)<\infty$. These are basic facts; see Sato \cite{S99} for proofs.
Let $ID_0=ID_0(\R^d)$ be the class of
$\mu\in ID$ with $A_{\mu}=0$. For any subclass $\mfr C$ of the class $ID$, 
let $\mfr C_0$ denote $\mfr C\cap ID_0$.
Sato \cite{S07}, p.\,85, introduced the dual $\mu'$ of 
$\mu\in ID_0$. But the naming of the dual of $\mu$ is
usually used for the distribution $\wt\mu$ that satisfies $\wt\mu(B)=\mu(-B)$ 
for $B\in \mcal B(\R^d)$, Borel sets in $\R^d$.  
So we call $\mu'$ the inversion of $\mu$ in this paper.

\begin{defn}\label{d0}
Let $\mu\in ID_0$. A distribution $\mu'\in ID_0$ is the {\it inversion} of $\mu$ if
\begin{gather}
\label{d2a1} \nu_{\mu'}(B)=\int_{\R^d\setminus \{0\}} 1_B(\iota (x))|x|^2 \nu_{\mu}
(dx),\quad B\in\mcal B(\R^d)\\
\label{d2a2} \gm_{\mu'}=-\gm_{\mu}+\int_{|x|=1} x\nu_{\mu}(dx),
\end{gather}
where $\iota(x)=|x|^{-2}x$, the geometric inversion of a point $x\in\R^d\setminus \{0\}$.
\end{defn}

For any subclass $\mfr C$ 
of $ID_0$, we will write $\mfr C'$ for the class $\{\mu'\cl \mu\in \mfr C\}$.
It is known that $\mu''=\mu$, that $\mu'$ has drift if and only if $\mu$ has mean, 
and that $\gm_{\mu'}^0=-m_{\mu}$.  Moreover, for $0<\al<2$, $\mu'$ is $(2-\al)$-stable if and only 
if $\mu$ is $\al$-stable; $\mu'$ is strictly $(2-\al)$-stable if and only 
if $\mu$ is strictly $\al$-stable.  These are shown in \cite{S07}.

The main subject of our study is 
the analysis of improper stochastic integrals with respect to L\'evy processes.  
Let $\{X_t^{(\rh)}\cl t\ges0\}$ be a L\'evy process on $\R^d$ such that 
$\law (X_1^{(\rh)})$, the distribution of $X_1^{(\rh)}$, equals $\rh$.
Consider improper stochastic integrals with respect to $\{X_t^{(\rh)}\}$
in two cases.
{\def\labelenumi{(\arabic{enumi})}
\begin{enumerate}
\item Let $0<c\les\infty$ and let $f(s)$ be a locally square-integrable function
on $[0,c)$ (that is, $\int_0^q f(s)^2 ds<\infty$ for $0<q<c$).
Then the stochastic integral $\int_0^q f(s)dX_s^{(\rh)}$ is 
defined for $0<q<c$ for all $\rh\in ID$.  
We say that the improper stochastic integral 
$\int_0^{c-} f(s)dX_s^{(\rh)}$ is definable if $\int_0^q f(s)dX_s^{(\rh)}$ 
is convergent in probability (or almost surely, or in law, equivalently) 
as $q\uar c$. Let $\Ph_f$ be the mapping from $\rh$ to
$\Ph_f\,\rh=\law \bigl( \int_0^{c-} f(s)dX_s^{(\rh)} \bigr)$.
Its domain $\mfr D(\Ph_f)$ is the class $\{\rh\in ID\cl \text{$\int_0^{c-} 
f(s)dX_s^{(\rh)}$ is definable}\}$.
\item Let $0<c<\infty$ and let $f(s)$ be a locally square-integrable function
on $(0,c]$.  Then $\int_p^c f(s)dX_s^{(\rh)}$ is 
defined for $0<p<c$ for all $\rh\in ID$.  
We say that the improper stochastic 
integral $\int_{0+}^c f(s)dX_s^{(\rh)}$ is definable if $\int_p^c f(s)dX_s^{(\rh)}$ 
is convergent in probability (or almost surely, or in law, equivalently) 
as $p\dar 0$.  Let $\Ph_f$ be the mapping from $\rh$ to
$\Ph_f\,\rh=\law \bigl( \int_{0+}^c f(s)dX_s^{(\rh)} \bigr)$ with 
$\mfr D(\Ph_f)=\{\rh\in ID\cl \text{$\int_{0+}^c
f(s)dX_s^{(\rh)}$ is definable}\}$.
\end{enumerate}}
\noindent The range  $\mfr R(\Ph_f)=\{\Ph_f\,\rh\cl 
\rh\in\mfr D(\Ph_f)\}$ is a subclass of $ID$. 
In any of the cases (1) and (2),  $\Ph_f$ is called a stochastic integral mapping as in 
\cite{S06a,S06b,S07,S10}.
If $c<\infty$ and $\int_0^c f(s)^2 ds<\infty$, then
$\int_0^{c-} f(s)dX_s^{(\rh)}=\int_{0+}^c f(s)dX_s^{(\rh)}=\int_0^c
f(s)dX_s^{(\rh)}$ for all $\rh\in ID$.
The analysis of $\mfr D(\Ph_f)$ is rather complicated if $\Ph_f=\Ph_{f_1}$ 
in case (1) with 
$f_1(s)\asymp s^{-1}$ as $s\to\infty$ (that is, there are positive constants $c_1$ 
and $c_2$ such that $c_1\les f_1(s) s\les c_2$ for all large $s$) or if
$\Ph_f=\Ph_{f_2}$ in case (2) with $f_2(s)\asymp s^{-1}$ as $s\dar 0$.  This
motivated us to introduce the inversion (or the dual) in \cite{S07} in order 
to reduce the study of $\Ph_{f_2}$ to that of $\Ph_{f_1}$.  More generally, let 
$f_1(s)$ and $f_2(s)$ be locally square-integrable on $[0,\infty)$ and
on $(0,c]$ with $c<\infty$, respectively.  
It is found in \cite{S07} that if the behavior of $f_1(s)$ decreasing to
$0$ as $s\to\infty$ and that of $f_2(s)$ increasing to $\infty$ as $s\dar0$
have some \lq\lq relation," then we can show that 
$(\mfr D(\Ph_{f_1})_0)'=\mfr D(\Ph_{f_2})_0$,
using the inversion.  This \lq\lq relation" exists if, with 
$0<\al<2$, $f_1(s)\asymp s^{-1/\al}$ as $s\to\infty$ and $f_2(s)\asymp s^{-1/(2-\al)}$
as $s\dar0$, or if $\log(1/f_1(s))\asymp s$  as $s\to\infty$ and $f_2(s)\asymp
s^{-1/2}$ as $s\dar0$.  
In these situations the study of $\mfr D(\Ph_{f_1})$ for $\int_0^{\infty-} f_1(s)
dX_s^{(\rh)}$ is equivalent to that of $\mfr D(\Ph_{f_2})$ for  
$\int_{0+}^c f_2(s)dX_s^{(\rh)}$.
But the relationship of $\mfr R(\Ph_{f_1})$ and 
$\mfr R(\Ph_{f_2})$ is more delicate and it has not been studied so far.

The range $\mfr R(\Ph_{f})$ of a stochastic integral mapping first appeared with
$f(s)=e^{-s}$ in the representation of the class of selfdecomposable distributions
by Wolfe \cite{W82} (see the historical notes in p.\,55 of \cite{RS03}).  
Jurek \cite{J85} proposed the problem to find, for each limit theorem, 
a function $f$ that describes 
the class of limit distributions for sequences of independent random variables
as $\mfr R(\Ph_f)$ (stochastic integral representation). 
Barndorff-Nielsen, Maejima, and Sato \cite{BMS06} described as
$\mfr R(\Ph_f)$ some well-known subclasses of $ID$.  
On the other hand, 
starting from the mappings $\Ph_f$ for some explicitly given 
families of $f$, Sato \cite{S06b}  
gave descriptions of $\mfr D(\Ph_f)$ and $\mfr R(\Ph_f)$. 
The two lines are intertwined and many studies have been made (\cite{BRT08,BT04,
J85,J88,J89,MMS10,MN09,MS09,S07,S10} etc.); thus numerous connections between 
stochastic integral mappings and subclasses of $ID$ have been found.

In this paper we continue to seek how to apply the inversion 
to the study of stochastic integral mappings.  
For the functions $f_1(s)$ and $f_2(s)$ above,
we want to find in what situation we can say that $(\mfr R(\Ph_{f_1})_0)'=
\mfr R(\Ph_{f_2})_0$.  For this purpose we will introduce the 
notion of the conjugate of a stochastic integral mapping.  
Most of the stochastic integral mappings $\Ph_f$ studied so far are
such that $f(s)$ is the inverse function of a function $g(t)$
defined by a positive function $h(u)$ as $g(t)=\int_t^b h(u)du$ for $t\in(a,b)$
with some $a$, $b$ satisfying $0\les a<b\les\infty$.
This situation drew attention in \cite{BRT08} and Section 7 of \cite{S07}; 
in the terminology of \cite{BRT08} this is $\Upsilon$-transformation
$\Upsilon^{\gamma}$ corresponding to an absolutely continuous measure 
$\gamma(du)=h(u)du$.   
In defining the conjugate, we consider only this case (except in 
Section 5) and, writing $g$ and $f$ as $g_h$ and $f_h$, let $\Ld_h$ denote
$\Ph_{f_h}$.  Under some condition, letting $h^*(u)=h(u^{-1})
u^{-4}$,  we will define the conjugate $\Ld_h^*=(\Ld_h)^*$ 
of $\Ld_h$ as 
$\Ld_h^*=\Ld_{h^*}$.  Then we will prove that $(\Ld_h^*)^*=\Ld_h$ and
that $\rh\in \mfr D(\Ld_h)_0$ and 
$\mu=\Ld_h \rh$ if and only if $\rh'\in \mfr D(\Ld_h^*)_0$ and 
$\mu'=\Ld_h^*(\rh')$. Thus $(\mfr R(\Ld_h)_0)'=\mfr R(\Ld_h^*)_0$.
Our next task is the study of $\Ld_h^*$, $\mfr D(\Ld_h^*)_0$, and
$\mfr R(\Ld_h^*)_0$ for several explicitly given mappings $\Ld_h$.
Specifically we will study $\bar \Ph_{p,\al}$ and $\Ld_{q,\al}$ 
of Sato \cite{S10} and $\Ps_{\al,\bt}$ of Maejima and Nakahara \cite{MN09};
the definitions of these mappings will be given in Section 4.

The contents of the sections are as follows.  Section 2 gives general 
properties and characterization of the inversion.  Defining the dilation 
$T_b\mu$, $b>0$, of a measure $\mu$ on $\R^d$ as
$(T_b\mu)(B)=\int_{\R^d} 1_B(bx)\mu(dx)$, $B\in\mcal B(\R^d)$, we find
the relation between dilation and inversion, which enables us to treat
easily semistable distributions under the action of inversion.
In Section 3 we introduce conjugates of stochastic integral mappings
and show their relations with inversion.  Here not only the usual
improper stochastic integrals but also their extension called 
essentially definable and their restriction called absolutely
definable introduced in Sato's papers are treated. 
Their definitions are recalled in Section 3.  These extension and restriction
are more manageable than $\Ph_f$ itself and give insight into the structure
of $\mfr D(\Ph_f)$ and $\mfr R(\Ph_f)$.  Section 4 is devoted to explicit 
description of domains and ranges of the stochastic integral mappings
$\bar \Ph_{p,\al}$, $\Ld_{q,\al}$, and $\Ps_{\al,\bt}$
and their conjugates.  Further, in the case of the conjugate of 
$\Ld_{1,\al}$, a connection with 
the class $L^{\la\al\ra *}$ defined by a new kind
of decomposability is given.  The class $L^{\la\al\ra *}$ is shown 
to be the class of inversions of $L^{\la\al\ra}$, where $L^{\la\al\ra}$ is
the class of $\al$-selfdecomposable distributions in Jurek \cite{J88,J89} and
Maejima and Ueda \cite{MU10a}.
In particular, $L^{\la 0\ra}$ is the class of selfdecomposable distributions.
The functions $f$ treated in Section 3 are positive and strictly decreasing.
Section 5 gives some results similar to Section 3 for $\Ph_f$ with a strictly
decreasing function $f$ taking
positive and negative values both.  Thus the class of inversions of
type $G$ distributions of Maejima and Rosi\'nski \cite{MR02} is treated.
We make in Section 6 a study of $\mfr R_{\infty}(\Ld_h^*)$, 
the limit of the nested classes $\mfr R((\Ld_h^*)^n)$, $n=1,2,\ldots$.
It is shown that $\mfr R_{\infty}(\Ld_h^*)_0=(\mfr R_{\infty}(\Ld_h)_0)'$,
which contributes to the study of the problem, treated in \cite{MS09,S11} 
and others, concerning what classes can appear as 
$\mfr R_{\infty}(\Ph_f)$ for general $f$.  

For L\'evy processes on $\R^d$
the weak version of Shtatland's theorem \cite{Sh65} concerning
$\lim_{s\dar0} s^{-1} X_s^{(\rh)}$ and the weak law of large numbers are
obtained from each other through inversion.  This remarkable application
of the inversion will be given in another paper \cite{SU11}.

\section{Properties and characterization of inversions}

First let us give a remark on the definition. In the two defining equations of
the inversion of $\mu\in ID_0$ in Definition \ref{d0}, the expression of 
\eqref{d2a2} depends on the choice of the integrand
in the L\'evy--Khintchine representation.
If we define $\gm^{\sharp}_{\mu}$ for $\mu\in ID_0$ by the representation
\begin{equation}\label{r1.1}
\wh\mu(z)=\exp\Bigl[ \int_{\R^d}\Bigl(e^{i\la z,x\ra}-1-\frac{i\la z,x\ra}{1+|x|^2}
\Bigr)\nu_{\mu}(dx)+i\la\gm^{\sharp}_{\mu},z\ra\Big],
\end{equation}
then \eqref{d2a2} is written as $\gm^{\sharp}_{\mu'}=-\gm^{\sharp}_{\mu}$\,, since
\[
\gm_{\mu}^{\sharp}=\gm_{\mu}-\int_{|x|\les1}\frac{|x|^2 x}{1+|x|^2} \nu_{\mu}(dx)
+\int_{|x|>1}\frac{x}{1+|x|^2}\nu_{\mu}(dx).
\]
Thus our definition of the inversion of $\mu$ is identical with the definition
of the dual of $\mu$ in Sato \cite{S07}.
If we define $\gm^{\sharp}_{\mu}$ for $\mu\in ID_0$ by
\begin{equation}\label{r1.1a}
\wh\mu(z)=\exp\Bigl[ \int_{\R^d}(e^{i\la z,x\ra}-1-i\la z,x\ra 1_{\{|x|<1\}}(x))
\nu_{\mu}(dx)+i\la\gm^{\sharp}_{\mu},z\ra\Big],
\end{equation}
then \eqref{d2a2} is expressed as 
$\gm^{\sharp}_{\mu'}=-\gm^{\sharp}_{\mu}-\int_{|x|=1}x\nu_{\mu}(dx)$. If
we define $\gm^{\sharp}_{\mu}$ for $\mu\in ID_0$ by
\begin{equation}\label{r1.1b}
\wh\mu(z)=\exp\Bigl[ \int_{\R^d}(e^{i\la z,x\ra}-1-i\la z,x\ra (1_{\{|x|<1\}}(x)
+\tfrac12 1_{\{|x|=1\}}(x))
\nu_{\mu}(dx)+i\la\gm^{\sharp}_{\mu},z\ra\Big],
\end{equation}
then \eqref{d2a2} is expressed as $\gm^{\sharp}_{\mu'}=-\gm^{\sharp}_{\mu}$. If
we define $\gm^{\sharp}_{\mu}$ for $\mu\in ID_0$ by
\begin{equation}\label{RR}
\wh\mu(z)=\exp\Bigl[ \int_{\R^d}(e^{i\la z,x\ra}-1-i\la z,x\ra c(x))
\nu_{\mu}(dx)+i\la\gm^{\sharp}_{\mu},z\ra\Big]
\end{equation}
with $c(x)=1_{\{|x|\les1\}}(x)+|x|^{-1} 1_{\{|x|>1\}}(x)$ as in Rajput and 
Rosinski \cite{RR89} and Kwapie\'n and Woyczy\'nski \cite{KW92}, then 
\eqref{d2a2} is expressed as 
$\gm^{\sharp}_{\mu'}=-\gm^{\sharp}_{\mu}+ \int_{|x|\les1}|x|x
\nu_{\mu}(dx)+\break \int_{|x|>1} |x|^{-1}x\nu_{\mu}(dx)$. 

\begin{prop}\label{p1a}
The inversion has the following properties.
\begin{enumerate}
\item Any $\mu\in ID_0$ has its inversion $\mu'\in ID_0$.
\item The inversion of $\mu'$ equals $\mu$, that is, $\mu''=\mu$.
\item $\int_{|x|\les1} |x|^{2-\al}\nu_{\mu'}(dx)=\int_{|x|\ges1} |x|^{\al}
\nu_{\mu}(dx)$ for $\al\in\R$.
\item $\mu'$ has drift if and only if $\mu$ has mean.
\item If $\mu$ has mean, then $\gm_{\mu'}^0=-m_{\mu}$.
\item If $\mu$ and $\mu_n$, $n=1,2,\ldots$, are in $ID_0$ and 
$\mu_n\to \mu$, then $\mu_n'\to \mu'$, where 
\lq\lq $\to$" denotes weak convergence.
\item $(\mu_1 *\mu_2)'=\mu_1' *\mu_2'$ for $\mu_1, \mu_2\in ID_0$.
\item $(\mu^s)'=(\mu')^s$ for $\mu\in ID_0$ and $s\ges0$, where $\mu^s$ denotes
the distribution with characteristic function $(\wh \mu(z))^s$.
\item If $\mu=\dl_c$ with $c\in\R^d$, then $\mu'=\dl_{-c}$, where $\dl_c$ 
denotes the $\dl$-distribution located at $c\in\R^d$.
\end{enumerate}
\end{prop}

Assertions (i)--(v) were already proved in \cite{S07}, but here we
repeat their proof for the convenience of readers.

\begin{proof}[Proof of Proposition \ref{p1a}] 
Given $\mu\in ID_0$, let $\nu^{\sharp}(B)$, $B\in\mcal B(\R^d)$,
be the right-hand side of \eqref{d2a1}. Then $\nu^{\sharp}(\{0\})=0$ and 
\begin{equation}\label{2.1}
\int_{\R^d}h(x)\nu^{\sharp}(dx)=\int_{\R^d\setminus\{0\}} h(\iota(x))|x|^2 \nu_{\mu}(dx)
\end{equation}
for any nonnegative measurable function $h(x)$.  Thus $\int_{\R^d}(|x|^2\land 1)
\nu^{\sharp}(dx)=\int_{\R^d}(|x|^2\land 1)\nu_{\mu}(dx)$.  Hence, (i) is true.
Moreover, it is readily proved that \eqref{2.1} is
valid for any $\R^d$-valued measurable function $h(x)$ on $\R^d$ satisfying 
$\int |h(x)|\nu^{\sharp}(dx)=\int |h(\iota(x))|\,|x|^2$ $\nu_{\mu}(dx)<\infty$.
To see (ii), note that
\[
\nu_{\mu''}(B)=\int_{\R^d\setminus \{0\}} 1_B(\iota (x))|x|^2 \nu_{\mu'}
(dx)=\nu_{\mu}(B)
\]
from \eqref{2.1} and that
\[
\gm_{\mu''}=-\gm_{\mu'}+\int_{|x|=1} x\nu_{\mu'}(dx)=\gm_{\mu}-\int_{|x|=1} x\nu_{\mu}(dx)
+\int_{|x|=1} x\nu_{\mu'}(dx)=\gm_{\mu}.\]
Assertion (iii) follows from \eqref{2.1}; (iv) follows from (iii) with $\al=1$.
If $\mu\in ID$ has drift, then $\gm_{\mu}^0=\gm_{\mu}-\int_{|x|\les1} x\nu_{\mu}(dx)$.
If $\mu\in ID$ has mean, then $m_{\mu}=\gm_{\mu}+\int_{|x|>1} x\nu_{\mu}(dx)$.
Hence we obtain (v) from (iv), noticing that
\[
\gm_{\mu'}^0=\gm_{\mu'}-\int_{|x|\les 1} x\nu_{\mu'}(dx)=-\gm_{\mu}+\int_{|x|=1} x
\nu_{\mu}(dx)-\int_{|x|\ges1} x\nu_{\mu}(dx)
=-m_{\mu}.
\]
To prove (vi) we use the expression \eqref{r1.1} in order to apply Theorem 8.7 
of \cite{S99}.  We write $f\in C_{\sharp}$ if $f$ is a
bounded continuous function from $\R^d$ to $\R$ vanishing on a neighborhood of $0$.
Let $\mu$ and $\mu_n$ be in $ID_0$.  In order that $\mu_n\to\mu$, it is necessary and 
sufficient that $\lim_{n\to\infty}\int_{\R^d} f(x)\nu_{\mu_n}(dx)=
 \int_{\R^d} f(x)\nu_{\mu}(dx)$ for $f\in C_{\sharp}$, 
$\lim_{\ep\dar0}\limsup_{n\to\infty} \int_{|x|\les\ep} |x|^2 \nu_{\mu_n}(dx)=0$, 
and $\lim_{n\to\infty}\gm_{\mu_n}^{\sharp} =\gm_{\mu}^{\sharp}$.
Now, assume that $\mu_n\to\mu$. Then, for $f\in C_\sharp$, we have
\[
\int f(x)\nu_{\mu'_n}(dx)=
\int_{|x|\les\ep} f(\iota(x)) |x|^2\nu_{\mu_n}(dx)+\int_{|x|>\ep}
 f(\iota(x)) |x|^2\nu_{\mu_n}(dx)=I_1+I_2,
\]
where $|I_1|$ is bounded by $\| f\|\int_{|x|\les\ep}|x|^2\nu_{\mu_n}(dx)$, and 
$I_2$ tends to $\int_{|x|>\ep}
 f(\iota(x)) |x|^2\nu_{\mu}(dx)$ as $n\to\infty$ if $\ep$ is chosen to satisfy
$\int_{|x|=\ep} \nu_{\mu}(dx)=0$. Hence, for $f\in C_\sharp$, 
\[
\int f(x)\nu_{\mu'_n}(dx)\to\int f(\iota(x))|x|^2\nu_{\mu}(dx)=\int f(x)\nu_{\mu'}(dx).
\]
Moreover, $\gm_{\mu'_n}^{\sharp}=-\gm_{\mu_n}^{\sharp}\to -\gm_{\mu}^{\sharp}
=\gm_{\mu'}^{\sharp}$, and
$\lim_{\ep\dar0}\limsup_{n\to\infty} \int_{|x|\les\ep} |x|^2 \nu_{\mu'_n}(dx)=0$,
since
$\int_{|x|\les\ep} |x|^2 \nu_n'(dx)=\int_{|x|\ges 1/\ep}\nu_{\mu_n}(dx)$.
Therefore $\mu_n'\to\mu'$. The converse follows from this by using (ii).

Since convolution induces addition in triplets, we have (vii).
Since $\mu^s$ has triplet $(0,s\nu_{\mu},s\gm_{\mu})$, we have (viii).
To see (ix), note that if $\mu=\dl_c$, then $\wh\mu(z)=e^{i\la c,z\ra}$,
so that $\gm_{\mu}=\gm_{\mu}^0=m_{\mu}=c$ and use (v).
\end{proof}

Let us give some characterization of the inversion. 

\begin{prop}\label{p2a'}
Suppose that $\mu\mapsto \mu^{\sharp}$ is a mapping from $ID_0$ into
$ID_0$ such that, for some $\al\in\R$,
\begin{equation}\label{p2a'1}
\nu_{\mu^{\sharp}}(B)=\int_{\R^d\setminus \{0\}} 1_B(\iota (x))|x|^{\al} \nu_{\mu}
(dx) \text{ for $\mu\in ID_0$ and $B\in\mcal B(\R^d)$}.
\end{equation}
Then $\al=2$.
\end{prop}

\begin{proof}
Since $\infty>\int_{|x|\les1} |x|^2 \nu_{\mu^{\sharp}}(dx)=\int_{|x|\ges1} |x|^{\al-2}
\nu_{\mu}(dx)$ for all $\mu\in ID_0$, we have $\al\les2$.
Since $\infty>\int_{|x|>1} \nu_{\mu^{\sharp}}(dx)=\int_{|x|<1} |x|^{\al}
\nu_{\mu}(dx)$ for all $\mu\in ID_0$, we have $\al\ges2$.
\end{proof}

In the following proposition let $ID_{0c}$ denote the class of $\mu\in ID_0$ with $\nu_{\mu}$ having 
compact support in $\R^d\setminus\{0\}$, that is, satisfying 
$\nu_{\mu}(\{|x|<a^{-1}\})=\nu_{\mu}(\{|x|>a\})=0$ for some $a>1$. 

\begin{prop}\label{p2a}
Suppose that $\mu\mapsto \mu^{\sharp}$ is a mapping from $ID_0$ into
$ID_0$ satisfying the following conditions:
\begin{enumerate}
\item \eqref{p2a'1} is true with $\al=2$;
\item $\mu^{\sharp \sharp}=\mu$ for $\mu\in ID_0$\,;
\item there is $k\in\R$ such that, for all $\mu\in ID_{0c}$,
$\gm_{\mu^{\sharp}}^0=k m_{\mu}$\,;
\item If $\mu$ and $\mu_n$, $n=1,2,\ldots$, are in $ID_0$ and 
$\mu_n\to \mu$, then $\mu_n^{\sharp}\to \mu^{\sharp}$.
\end{enumerate}
Then $k=-1$ and $\mu^{\sharp}=\mu'$ for $\mu\in ID_0$.
\end{prop}

\begin{proof}
If $\mu\in ID_{0c}$, then $\mu^{\sharp}\in ID_{0c}$ and $\mu$ and $\mu^{\sharp}$
have drift and mean.
If $\mu\in ID_{0c}$, then the identity $\gm_{\mu^{\sharp}}^0=k m_{\mu}$
is written as 
\[
\gm_{\mu^{\sharp}}-\int_{|x|\les1} x\nu_{\mu^{\sharp}}=k\left( \gm_{\mu}+
\int_{|x|>1}x\nu_{\mu}(dx)\right),
\]
that is,
\[
\gm_{\mu^{\sharp}}=k\gm_{\mu}+(k+1)\int_{|x|>1} x\nu_{\mu}(dx)+\int_{|x|=1}
x\nu_{\mu}(dx),
\]
since \eqref{2.1} is true with $\nu_{\mu^{\sharp}}$ in place of $\nu^{\sharp}$.
Hence, if $\mu\in ID_{0c}$, then
\[
\gm_{\mu^{\sharp\sharp}}=k^2 \gm_{\mu}+k(k+1)\int_{|x|>1} x\nu_{\mu}(dx)+(k+1)\int_{|x|=1}
x\nu_{\mu}(dx)+(k+1)\int_{|x|<1}x\nu_{\mu}(dx),
\]
which combined with condition (ii) says that
\[
(1-k^2)\gm_{\mu}=k(k+1)\int_{|x|>1} x\nu_{\mu}(dx)+(k+1)\int_{|x|\les1}x\nu_{\mu}
(dx).
\]
This is absurd if $k\neq -1$.  Indeed, if $k^2\neq 1$, then this would mean that
\[
\gm_{\mu}=\frac{1+k}{1-k^2} \left( k\int_{|x|>1} x\nu_{\mu}(dx)+
\int_{|x|\les1} x\nu_{\mu}(dx)\right)
\]
for all $\mu\in ID_{0c}$; if $k=1$, then this would mean that
$0=2\int_{\R^d} x\nu_{\mu}(dx)$ for all $\mu\in ID_{0c}$.
Therefore $k=-1$. Hence $\gm_{\mu^{\sharp}}-\int_{|x|\les1} x\nu_{\mu^{\sharp}}(dx)=-\gm_{\mu}
-\int_{|x|>1}x\nu_{\mu}(dx)$, that is, $\gm_{\mu^{\sharp}}=\gm_{\mu'}$ 
for all $\mu\in ID_{0c}$. Hence $\mu^{\sharp}
=\mu'$ for all $\mu\in ID_{0c}$. 
Approximating a general $\mu\in ID_0$ by $\mu_n\in ID_{0c}$ 
and using condition (iv) together with 
Proposition \ref{p1a} (vi), we obtain $\mu^{\sharp}
=\mu'$ for all $\mu\in ID_0$.
\end{proof}

The dilation $T_b\mu$ of a measure $\mu$ on $\R^d$ is defined in Section 1. 

\begin{prop}\label{p3}
Let $b>0$.  Then $(T_b\mu)'=(T_{b^{-1}}(\mu'))^{b^2}$ for $\mu\in ID_0$.
\end{prop}

\begin{proof}
We have
\begin{equation}\label{p3.1}
\nu_{T_b\mu}=T_b \nu_{\mu}\quad\text{and}\quad\gm_{T_b\mu}=b\gm_{\mu}
\begin{cases}+b\int_{1<|x|\les b^{-1}} x\nu_{\mu}(dx)\quad &\text{if }b<1\\
-b\int_{b^{-1}<|x|\les 1} x\nu_{\mu}(dx)\quad &\text{if }b>1.
\end{cases}
\end{equation}
Assume that $b>1$. Then
\begin{align*}
\nu_{(T_b\mu)'}(B)&=b^2\int 1_{bB}(\iota(x)) |x|^2 \nu_{\mu}(dx)=
b^2\nu_{\mu'}(bB)=b^2(T_{b^{-1}}(\nu_{\mu'}))
(B),\\
\gm_{(T_b\mu)'}&=-\gm_{T_b \mu}+b\int_{|x|=b^{-1}}x\nu_{\mu}(dx)
=-b\gm_{\mu} +b\int_{b^{-1}\les |x|\les 1} x\nu_{\mu}(dx)\\
&=b\gm_{\mu'} +b\int_{1<|x|\les b}x\nu_{\mu'}(dx)=b^2 \gm_{T_{b^{-1}}(\mu')}.
\end{align*}
This proves the assertion for $b>1$.  This result and Proposition \ref{p1a} (viii)
yield the assertion for $0<b<1$.  It is trivial for $b=1$.
\end{proof}

Let $0<\al\les2$, $b>1$, and $\mu\in ID$. 
We say that $\mu$
is $\al$-semistable [resp.\ strictly $\al$-semistable] with a span $b$ 
if  $\mu^{b^{\al}}=(T_b\mu)*\dl_c$
for some $c\in\R^d$ [resp.\ $\mu^{b^{\al}}=T_b\mu$].
We say that $\mu$ is $\al$-stable [resp.\ 
strictly $\al$-stable] if, for all $b>1$, $\mu$ is $\al$-semistable [resp.\  
strictly $\al$-semistable] with a span $b$. Thus any trivial distribution 
(that is, $\dl$-distribution) is $\al$-stable for $0<\al\les 2$.

The following theorem gives further remarkable properties of the inversion.
If $\mu$ is $\al$-semistable with $0<\al<2$, then $\mu\in ID_0$. 
Assertions (iii) and (iv) were shown in \cite{S07}, but we will
give a new proof.

\begin{thm}\label{t0}
Let $0<\al<2$, $b>1$, and $\mu\in ID_0$.
\begin{enumerate}
\item $\mu'$ is $(2-\al)$-semistable with a span $b$ if and only
if $\mu$ is $\al$-semistable with a span $b$.
\item $\mu'$ is strictly $(2-\al)$-semistable with a span $b$ 
if and only if $\mu$ is strictly $\al$-semistable with a span $b$.
\item $\mu'$ is $(2-\al)$-stable if and only
if $\mu$ is $\al$-stable.
\item $\mu'$ is strictly $(2-\al)$-stable 
if and only if $\mu$ is strictly $\al$-stable.
\end{enumerate}
\end{thm}

\begin{proof}
In general we have, for $b,b_1,b_2>0$, 
$T_b(\mu_1*\mu_2)=(T_b\mu_1)*(T_b\mu_2)$, $T_b(\mu^s)=(T_b\mu)^s$, 
$T_{b_2}(T_{b_1}\mu)=T_{b_2 b_1}\mu$, and $T_b(\dl_c)=\dl_{bc}$.
Let us prove assertion (i). Let $0<\al<2$ and $b>1$.
Assume that $\mu$ is $\al$-semistable with a span $b$. 
Then $\mu^{b^{\al}}=(T_b\mu)*\dl_c$, and hence
$(T_{b^{-1}}\mu)^{b^{\al}}=
\mu*\dl_{b^{-1}c}$, that is, $\mu=(T_{b^{-1}}\mu)^{b^{\al}}*\dl_{-b^{-1}c}$.
This gives
$\mu^{b^{-\al}}=(T_{b^{-1}}\mu)*\dl_{-b^{-\al-1}c}$. 
Now go to inversions and use Propositions \ref{p1a} and \ref{p3}.
Then $(\mu')^{b^{-\al}}=(T_b(\mu'))^{b^{-2}}*\dl_{b^{-\al-1}c}$.
Hence $(\mu')^{b^{2-\al}}=(T_b(\mu'))*\dl_{b^{1-\al}c}$, and
$\mu'$ is $(2-\al)$-semistable with a span $b$. The converse is also
proved from this, since $\mu''=\mu$. Thus (i) is true.
Assertion (ii) is shown by letting $c=0$ in the argument above.
Assertions (iii) and (iv) are automatic from (i) and (ii).
Another proof of this theorem can be given by using the characterization of L\'evy
measures of (strictly) semistable and stable distributions in \cite{S99}.
\end{proof}
 
Any $\sg$-finite measure $\nu$ on $\mathbb{R}^d$ with $\nu(\{0\})=0$
has two decompositions.
Let $S=\{\xi\in\R^d\cl |\xi|=1\}$, the unit sphere in $\R^d$, and 
$\Rpl=(0,\infty)$.  (1)  There are a $\sg$-finite measure $\ld$ on $S$ 
with $\ld(S)\ges0$ and
a measurable family $\{\nu_{\xi}\colon \xi\in S\}$
 of $\sg$-finite measures on $\Rpl$ with $\nu_{\xi}(\Rpl)>0$ such that
$\nu(B)=\int_{S}\ld(d\xi)\int_{\Rpl} 1_B(r\xi)\nu_{\xi}(dr)$, 
$B\in\mcal B(\R^d)$.  The pair $(\ld(d\xi),\nu_{\xi}(dr))$ is called
a radial decomposition of $\nu$.  It is unique in the sense that,
if $(\ld^1(d\xi),\nu_{\xi}^1(dr))$ and $(\ld^2(d\xi),\nu_{\xi}^2(dr))$ are both 
radial decompositions of $\nu$, then, for some positive, finite, 
measurable function $c(\xi)$ on $S$, we have $c(\xi)\ld^2(d\xi)=
\ld^1(d\xi)$ and $\nu_{\xi}^2(dr)=c(\xi)\nu_{\xi}^1(dr)$ for 
$\ld^1$-a.\,e.\ $\xi\in S$.
(2)  There are a $\sg$-finite measure $\bar\nu$ on $\Rpl$ with 
$\bar\nu(\Rpl)\ges0$ and 
a measurable family $\{\ld_r\colon r\in \Rpl\}$ of $\sg$-finite measures on
$S$ with $\ld_r(S)>0$  such that $\nu(B)=\int_{\Rpl}\bar\nu(dr)
\int_{S} 1_B(r\xi)\ld_r(d\xi)$,  $B\in\mcal B(\R^d)$.
The pair $(\bar\nu(dr),\ld_r(d\xi))$ is called a spherical decomposition
of $\nu$.  It is unique in a sense similar to in (1).
See Sato \cite{S10}, pp.\,27--28 for details.

\begin{ex}\label{2e1}
Suppose that $\mu$ is $1$-stable. Then $\mu'=\mu*\dl_{-2\gm_{\mu}}$. Indeed,
$\nu_{\mu}$ has a radial decomposition $(\ld(d\xi),r^{-2} dr)$ and hence
\[
\nu_{\mu'}(B)=\int_S \ld(d\xi)\int_{\Rpl} 1_B(r^{-1}\xi) dr=
\int_S \ld(d\xi)\int_{\Rpl} 1_B(r\xi) r^{-2} dr=\nu_{\mu}(B)
\]
and $\gm_{\mu'}=-\gm_{\mu}$.  Thus
\[
\wh{\mu'}(z)=\exp\Bigl[ \int_{\R^d}(
e^{i\langle z,x\rangle}-1-i\langle z,x\rangle 1_{\{|x|\les 1\}}(x) )\nu_{\mu}(dx)
-i\langle\gamma_{\mu},z\rangle\Bigr].
\]
If $\gm_{\mu}=0$, then $\mu$ is self-inversion, that is, $\mu'=\mu$.
If $\nu_{\mu}\neq0$, then $\ld\neq0$ and $(r^{-2}dr,\ld)$ is  
a spherical decomposition of $\nu_{\mu}$ at the same time.
\end{ex}

\begin{ex}\label{2e2}
Let $d=1$. Let $\mu$ be Poisson distribution with mean $m>0$.  Then $\nu_{\mu}=m\dl_1$
and hence $\nu_{\mu'}=m\dl_1=\nu_{\mu}$.  Thus $\mu'$ is translated Poisson distribution.
Since $\mu$ has drift $0$ and mean $m$, $\mu'$ has mean $0$ and drift $-m$
and $\mu'=\mu*\dl_{-m}$.  Note that $\mu*\dl_{-m/2}$ is self-inversion. 
\end{ex}


\section{Conjugates of stochastic integral mappings}

We introduce Condition (C) on a function $h$ 
and define the conjugate of a stochastic integral 
mapping associated with a function $h$ satisfying this condition. 
Then we give main results on the connection
between the conjugate and the inversion.

\begin{defn}\label{d1}
A function $h$ is said to satisfy \emph{Condition} (C) 
if there are $a_h$ and $b_h$ with $0\les a_h<b_h\les \infty$ such that $h$
is defined on $(a_h,b_h)$, positive, and measurable, and 
at least one of the following is true:
\begin{gather}\label{d1.2}
\int_{a_h}^{b_h} h(u) u^2 du<\infty,\\
\int_{a_h}^{b_h} h(u)du<\infty. \label{d1.3}
\end{gather}
A function $h$ satisfying Condition (C) with 
\eqref{d1.2} [resp.\ \eqref{d1.3}] is said to 
satisfy (C$_1$) [resp.\ (C$_2$)].
\end{defn}

\begin{defn}\label{d1a}
Let $h$ be a function satisfying Condition $(\mrm{C})$. Define a function $h^*$
as $a_{h^*}=1/b_h$, $b_{h^*}=1/a_h$ (letting $1/0=\infty$
and $1/\infty=0$), and
\begin{equation}\label{p1.1}
h^*(u)=h(u^{-1}) u^{-4}, \qquad u\in(a_{h^*}, b_{h^*}).
\end{equation}
\end{defn}

\begin{prop}\label{p2}
If $h$ satisfies Condition $(\mrm{C})$, then 
$h^*$ satisfies Condition $(\mrm{C})$.  If $h$ satisfies $(\mrm{C}_1)$, 
then $h^*$ satisfies $(\mrm{C}_2)$.
If $h$ satisfies $(\mrm{C}_2)$, then $h^*$ satisfies $(\mrm{C}_1)$.
Moreover, $(h^*)^*=h$.
\end{prop}

\begin{proof} Notice that
\begin{align*}
\int_{a_{h^*}}^{b_{h^*}} h^*(u)u^2du&=\int_{1/b_h}^{1/a_h}h(u^{-1})
u^{-4}u^2du=\int_{a_h}^{b_h} h(v)dv,\\
\int_{a_{h^*}}^{b_{h^*}} h^*(u)du&=\int_{1/b_h}^{1/a_h}h(u^{-1})
u^{-4}du=\int_{a_h}^{b_h} h(v)v^2dv.
\end{align*}
Then the assertions on $h^*$ follow from the properties of $h$.
The relation $(h^*)^*=h$ is obvious.
\end{proof}

For each function $h(u)$ satisfying Condition (C),  let
\begin{equation}\label{1}
g_h(t)=\int_t^{b_h} h(u)du,\qquad t\in (a_h,b_h).
\end{equation}
Then $g_h(t)$ is a strictly decreasing, continuous function with $g_h(b_h -)=0$. 
Let $c_h=g_h(a_h +)$. Define $f_h(s)$ as
\[
s=g_h(t)\text{ with }a_h<t< b_h\quad\Leftrightarrow\quad t=f_h(s)\text{ with }
0< s<c_h.
\]
Then $f_h(s)$ is a strictly decreasing, continuous function with $f_h(0+)=b_h$
and\linebreak $f_h(c_h -)=a_h$, and 
\begin{equation}\label{2}
\int_u^{c_h} f_h(s)^2 ds <\infty, \qquad u\in (0,c_h),
\end{equation}
since
\[
\int_u^{c_h} f_h(s)^2 ds =\int_{f_h(u)}^{a_h} t^2 dg_h(t)=\int_{a_h}^{f_h(u)} h(t) t^2 dt.
\]
We have
\begin{gather}\label{3}
\int_0^{c_h} f_h(s)^2 ds <\infty \qquad \text{if $h$ satisfies $(\mrm{C}_1)$},\\
c_h<\infty \qquad\text{if $h$ satisfies $(\mrm{C}_2)$}.\label{4}
\end{gather}
Define a stochastic integral mapping $\Ph_{f_h}$ as $\Ph_f$ in Section 1 with
$f=f_h$.  Indeed, we have, for $\rh\in\mfr D(\Ph_{f_h})$,
\begin{align*}
\Ph_{f_h}\rh&=\law\left(\int_{0}^{c_h -} f_h(s)dX_s^{(\rh)}\right)
\qquad \text{if $h$ satisfies $(\mrm{C}_1)$},\\
\Ph_{f_h}\rh&=\law\left(\int_{0+}^{c_h} f_h(s)dX_s^{(\rh)}\right) 
\qquad\text{if $h$ satisfies $(\mrm{C}_2)$},
\intertext{and}
\Ph_{f_h}\rh&=\law\left(\int_{0}^{c_h} f_h(s)dX_s^{(\rh)}\right)
\qquad \text{if $h$ satisfies $(\mrm{C}_1)$ and $(\mrm{C}_2)$}.
\end{align*}

\begin{defn}\label{newdef}
If $h$ is a function satisfying Condition (C), then $\Ph_{f_h}$ is written 
as $\Ld_h$.   We call
the stochastic integral mapping $\Ld_{h^*}$ 
the {\it conjugate} of $\Ld_h$ and write $\Ld_{h^*}$ as $\Ld_h^*$.
Thus $\Ld_h^*=\Ld_{h^*}=\Ph_{f_{h^*}}$. 
\end{defn}

\begin{prop}\label{newprop}
The conjugate of $\Ld_h^*$ coincides with $\Ld_h$.
\end{prop}

\begin{proof}
This is a direct consequence of Proposition \ref{p2}.
\end{proof}

In general, given a function $h$ satisfying Condition (C),  we write 
$a$, $b$, $c$, $g$, $f$, $a_*$, $b_*$, $c_*$, $g_*$, and $f_*$ for
$a_h$, $b_h$, $c_h$, $g_h$, $f_h$, $a_{h^*}$, $b_{h^*}$, $c_{h^*}$, 
$g_{h^*}$, and $f_{h^*}$, respectively, if no confusion arises.

In the study of a stochastic integral mapping $\Ph_f$ it is important to use some
extension and some restriction of $\Ph_f$, because they are more manageable than
$\Ph_f$ itself and give information on the structure of the domain and the range.
Suppose that $h$ satisfies $(\mrm{C}_1)$ [resp.\ $(\mrm{C}_2)$].
A distribution $\rh\in ID$ is in $\mfr D(\Ld_h)$ if and only if 
$\int_0^q \log\wh\rh(f_h(s)z)ds$ [resp.\ $\int_p^{c_h} \log\wh\rh(f_h(s)z)ds$] 
is convergent as 
$q\uar c_h$ [resp.\ $p\dar 0$] for every $z\in\R^d$ (\cite{S06b}, p.\,51).  
We say that 
$\Ld_h \rh$ is absolutely definable if $\int_0^{c_h} |\log\wh\rh(f_h(s)z)|ds
<\infty$ for every $z\in\R^d$. We say that $\Ld_h \rh$ is essentially definable if,
for some $\R^d$-valued function $k$ on $[0,c_h)$ [resp.\ $(0,c_h]$] and some
$\R^d$-valued random variable $Y$, $\int_0^q f_h(s)dX_s^{(\rh)}-k(q)$ 
[resp.\ $\int_p^{c_h} f_h(s)dX_s^{(\rh)}-k(p)$] converges to $Y$ in probability
as $q\uar c_h$ [resp.\ $p\dar 0$].  Define
\begin{align*}
\mfr D^0 (\Ld_h)&=\{\rh\in ID\colon \text{$\Ld_h \rh$ is absolutely definable}
\},\\
\mfr D^{\mrm{e}} (\Ld_h)&=\{\rh\in ID\colon \text{$\Ld_h \rh$ is essentially definable} 
\},\\
\mfr R^0 (\Ld_h)&=\{ \mu=\Ld_h \rh\colon \rh\in \mfr D^0 (\Ld_h)\},\\
\mfr R^{\mrm{e}} (\Ld_h)&=\{ \mu=\law (Y)\colon \text{$\rh\in \mfr D^{\mrm{e}} (\Ld_h)$ 
and all $k$ and $Y$ that can be chosen}\\
&\qquad\text{in the definition of essential definability of $\Ld_h \rh$}\}.
\end{align*}
Then $\mfr D^0 (\Ld_h)\subset \mfr D (\Ld_h) \subset \mfr D^{\mrm{e}} (\Ld_h)$
and $\mfr R^0 (\Ld_h)\subset \mfr R (\Ld_h) \subset \mfr R^{\mrm{e}} (\Ld_h)$.
The condition for $\rh$ or $\mu$ in $ID$ to belong to these classes can be
described in terms of their triplets (see \cite{S06b,S07,S10}).

\begin{thm}\label{t1}
Let $h$ be a function satisfying Condition $(\mrm{C})$.  Consider $\Ld_h$ and
its conjugate $\Ld_h^*$.  Let $\rh\in ID_0$.  Then
\begin{equation}\label{t1.1}
\rh\in \mfr D(\Ld_h) \quad\text{and}\quad \Ld_h\rh=\mu
\end{equation}
if and only if
\begin{equation}\label{t1.2}
\rh'\in \mfr D(\Ld_h^*) \quad\text{and}\quad \Ld_h^*\rh'=\mu'
\end{equation}
Furthermore,
\begin{align}\label{t1.3}
\mfr D (\Ld_h^*)_0&= (\mfr D (\Ld_h)_0)',\\
\label{t1.4}
\mfr D^{\mrm{e}} (\Ld_h^*)_0&= (\mfr D^{\mrm{e}} (\Ld_h)_0)',\\
\label{t1.5}
\mfr D^0 (\Ld_h^*)_0&= (\mfr D^0 (\Ld_h)_0)',\\
\label{t1.6}
\mfr R (\Ld_h^*)_0&= (\mfr R (\Ld_h)_0)',\\
\label{t1.7}
\mfr R^{\mrm{e}} (\Ld_h^*)_0&= (\mfr R^{\mrm{e}} (\Ld_h)_0)',\\
\label{t1.8}
\mfr R^0 (\Ld_h^*)_0&= (\mfr R^0 (\Ld_h)_0)'.
\end{align}
\end{thm}

\begin{proof}  
{\it Step 1}. 
Given $\rh\in ID_0$, assume \eqref{t1.1}. Then, $\mu\in ID_0$.
In order to prove \eqref{t1.2}, it is enough to show that
\begin{gather}
\label{t1p2}
\int_0^{c_*} ds\int_{\R^d} (|f_*(s)x|^2\land 1)
\nu_{\rh'}(dx)<\infty,\\
\label{t1p1}
\nu_{\mu'}(B)=\int_0^{c_*} ds\int_{\R^d} 1_B(f_*(s)x)\nu_{\rh'}(dx)\qquad
\text{for }B\in\mcal B(\R^d\setminus\{0\}),\\
\label{t1p3}
\gm_{\mu'}=\int_{0+}^{c_* -} f_*(s)ds \left[ \gm_{\rh'}+
\int_{\R^d} x(1_{\{f_*(s)|x|\les1\}}-1_{\{|x|\les 1\}})\nu_{\rh'}(dx)
\right]
\end{gather}
(see Theorems 3.5 and 3.10 of \cite{S07} or Proposition 3.18 of \cite{S10}).
It follows from \eqref{t1.1} that \eqref{t1p2}--\eqref{t1p3} hold for 
$\mu$, $\rh$, $f(s)$, and $c$ in place of $\mu'$, $\rh'$, $f_*(s)$, and $c_*$.
Thus
\[
\nu_{\mu}(B)=\int_a^b h(t)dt \int_{\R^d} 1_B(tx)\nu_{\rh}(dx).
\]
Using \eqref{d2a1}, we have
\begin{align*}
\nu_{\mu'}(B)&=\int_a^b h(t)dt \int_{\R^d} 1_B(t^{-1}|x|^{-2} x)t^2 |x|^2
\nu_{\rh}(dx)\\
&=\int_a^bh(t)dt \int_{\R^d} 1_B(t^{-1}x)t^2 \nu_{\rh'}(dx)
=\int_{a_*}^{b_*} h^*(u)du\int_{\R^d} 1_B(ux)\nu_{\rh'}(dx).
\end{align*}
Hence \eqref{t1p1} is true.
We have \eqref{t1p2} from \eqref{t1p1}, since
$\int(|x|^2\land 1)\nu_{\mu'}(dx)<\infty$. 
Moreover,
\[
\int_{|x|=1}\nu_{\mu}(dx)=\int_0^c ds\int_{\R^d} 1_{\{f(s)|x|=1\}} \nu_{\rh}(dx)
=\int_{\R^d}\nu_{\rh}(dx) \int_0^c 1_{\{f(s)=|x|^{-1}\}} ds=0,
\]
as $f(s)$ is strictly decreasing. Hence, from \eqref{d2a2} and from \eqref{t1p3}
for $\mu$,
\[
\gm_{\mu'}=-\gm_{\mu}=-\int_{0+}^{c-} f(s)ds \left[ \gm_{\rh}+\int_{\R^d} 
x(1_{\{f(s)|x|\les1\}}-1_{\{|x|\les1\}})\nu_{\rh}(dx)\right].
\]
Hence
\begin{align*}
\gm_{\mu'}&=\int_{0+}^{c-} f(s)ds \left[ \gm_{\rh'}-\int_{|x|=1}x\nu_{\rh'}(dx)-
\int_{\R^d} x(1_{\{f(s)|x|^{-1}\les1\}}-1_{\{|x|^{-1}\les1\}})\nu_{\rh'}(dx)
\right]\\
&=\int_{0+}^{c-} f(s)ds \left[ \gm_{\rh'}+
\int_{\R^d} x(1_{\{|x|>1\}}-1_{\{|x|\ges f(s)\}})\nu_{\rh'}(dx)
\right]\\
&=\int_{a+}^{b-} t\,h(t)dt \left[ \gm_{\rh'}+
\int_{\R^d} x(1_{\{|x|>1\}}-1_{\{|x|\ges t\}})\nu_{\rh'}(dx)
\right]\\
&=\int_{a_* +}^{b_* -} u h^*(u) du \left[ \gm_{\rh'}+
\int_{\R^d} x(1_{\{|x|>1\}}-1_{\{|x|\ges u^{-1}\}})\nu_{\rh'}(dx)
\right]\\
&=\int_{0+}^{c_* -} f_*(s)ds \left[ \gm_{\rh'}+
\int_{\R^d} x(1_{\{|x|>1\}}-1_{\{f_*(s)|x|\ges 1\}})\nu_{\rh'}(dx)
\right]\\
&=\int_{0+}^{c_* -} f_*(s)ds \left[ \gm_{\rh'}+
\int_{\R^d} x(1_{\{f_*(s)|x|<1\}}-1_{\{|x|\les 1\}})\nu_{\rh'}(dx)
\right].
\end{align*}
Since
\[
\int_{0+}^{c_* -} f_*(s)ds \int_{\R^d}|x|1_{\{f_*(s)|x|=1\}} \nu_{\rh'}(dx)
=\int_{\R^d} |x|\nu_{\rh'}(dx) \int_{0+}^{c_* -} f_*(s) 1_{\{f_*(s)=|x|^{-1}\}}
ds=0
\]
as $f_*(s)$ is strictly decreasing, we obtain \eqref{t1p3}.
Thus \eqref{t1.2} holds.
That is, \eqref{t1.1} implies \eqref{t1.2}.
Now \eqref{t1.2} implies \eqref{t1.1} automatically, since we have 
$\rh''=\rh$, $\mu''=\mu$, and Proposition \ref{p2}. We also obtain \eqref{t1.3}
and \eqref{t1.6}.

{\it Step 2}.  Let us prove \eqref{t1.4} and \eqref{t1.7}.  Assume that 
$\rh\in \mfr D^{\mrm{e}} (\Ld_h)_0$. Let $k$ and $Y$ be those in the 
definition of essential definability.  Let $\mu=\law (Y)$. 
Then $\mu\in ID_0$ and the analogue of \eqref{t1p1} for $\mu$ holds.
As in Step 1, we obtain \eqref{t1p2} and \eqref{t1p1}.
Hence, by Theorem 3.6 of \cite{S07} or Proposition 3.18 of \cite{S10},
$\rh'\in \mfr D^{\mrm{e}} (\Ld_h^*)_0$.
If $h$ satisfies (C$_2$) [resp.\ (C$_1$)], then $h^*$ satisfies 
(C$_1$) [resp.\ (C$_2$)], 
and we obtain $\mu'\in \mfr R^{\mrm{e}} (\Ld_h^*)$ from \eqref{t1p1}
and  Proposition 3.27 of \cite{S10} [resp.\ an analogue of Proposition 
3.27 of \cite{S10} for the $\int_{0+}^c$ type integral in $ID_0$]. 
Thus \eqref{t1.4} and \eqref{t1.7} are
proved with $=$ replaced by $\supset$. The converse inclusions
automatically follow from this and (ii) of Proposition \ref{p1a}. 
Hence \eqref{t1.4} and \eqref{t1.7} are true.

{\it Step 3}.  Let us prove \eqref{t1.5} and \eqref{t1.8}. Assume that 
$\rh\in \mfr D^0 (\Ld_h)_0$. Then, by Proposition 2.3 of \cite{S06c}
or Proposition 3.18 of \cite{S10}, 
\[
\int_0^c f(s)ds \left| \gm_{\rh}+
\int_{\R^d} x(1_{\{f(s)|x|\les 1\}}-1_{\{|x|\les 1\}})\nu_{\rh}(dx)
\right| <\infty.
\]
The outer integral equals
\[
\int_0^{c_*} f_*(s)ds \left| \gm_{\rh'}+
\int_{\R^d} x(1_{\{f_*(s)|x|\les 1\}}-1_{\{|x|\les 1\}})\nu_{\rh'}(dx)
\right| 
\]
by the same calculation as in Step 1.  Since we already have \eqref{t1p2},
this shows that $\rh'\in \mfr D^0 (\Ld_h^*)_0$.
Let $\mu=\Ld_h\rh$.  Then $\mu'=\Ld_h^* \rh'$ by the result of
Step 1.  Hence $\mu'\in \mfr R^0(\Ld_h^*)$.
Hence  \eqref{t1.5} and \eqref{t1.8} are
proved with $=$ replaced by $\supset$. Then the converse inclusions are
automatic.
\end{proof}

In view of Theorem \ref{t1}, the relations of the domains 
and the ranges of
$\Ph_f$ with their restrictions to $ID_0$ are of interest.

\begin{prop}\label{p4}
Let $\Ph_f$ be a stochastic integral mapping.  Then the classes $\mfr D$, 
$\mfr D^{\mrm{e}}$, $\mfr D^0$, $\mfr R$, $\mfr R^{\mrm{e}}$,  and\/ 
$\mfr R^0$ of\/ $\Ph_f$ are closed under convolution.
\end{prop}
\begin{proof}
For $\mfr D$ and $\mfr R$ the assertion follows from the fact that if 
$\rh_1,\rh_2\in \mfr D(\Ph_f)$, then 
$\rh_1*\rh_2\in \mfr D(\Ph_f)$ and $\Ph_f(\rh_1*\rh_2)=(\Ph_f\rh_1)*
(\Ph_f\rh_2)$. It is in Propositions 3.18 and 3.20 of \cite{S10} and their
analogue for improper stochastic integrals of $\int_{0+}^c$ type in 
Section 3 of \cite{S07}. The other assertions are derived similarly.
\end{proof}

Let $\mfr S_2=\{\rh\in ID\cl \text{$2$-stable} \}
=\{\rh\in ID\cl \text{Gaussian}\}$ and let 
$\mfr S_2^0=\{\rh\in \mfr S_2\cl m_{\rh}=0\}$.

\begin{prop}\label{p5}
Let $\Ph_f$ be as in {\rm (1)} and {\rm (2)} in Section 1.
\begin{enumerate}
\item If\/ $0<\int_0^c f(s)^2 ds<\infty$, then 
\begin{gather*}
\mfr D(\Ph_f)=\{\rh_1*\rh_0\cl \rh_1\in\mfr S_2^0,\, \rh_0\in \mfr D(\Ph_f)_0\},\\
\mfr D^{\mrm{e}}(\Ph_f)=\{\rh_1*\rh_0\cl \rh_1\in\mfr S_2^0,\, 
\rh_0\in \mfr D^{\mrm{e}}(\Ph_f)_0\},\\
\mfr D^0(\Ph_f)=\{\rh_1*\rh_0\cl \rh_1\in\mfr S_2^0,\, \rh_0\in \mfr D^0(\Ph_f)_0\},\\
\mfr R(\Ph_f)=\{\mu_1*\mu_0\cl \mu_1\in\mfr S_2^0,\, \mu_0\in \mfr R(\Ph_f)_0\},\\
\mfr R^{\mrm{e}}(\Ph_f)=\{\mu_1*\mu_0\cl \mu_1\in\mfr S_2^0,\,
\mu_0\in \mfr R^{\mrm{e}}(\Ph_f)_0\},\\
\mfr R^0(\Ph_f)=\{\mu_1*\mu_0\cl \mu_1\in\mfr S_2^0,\, \mu_0\in \mfr R^0(\Ph_f)_0\}.
\end{gather*}
\item  If $\int_0^c f(s)^2 ds=\infty$, then $\mfr D$, 
$\mfr D^{\mrm{e}}$, $\mfr D^0$, $\mfr R$, $\mfr R^{\mrm{e}}$,  and 
$\mfr R^0$ of $\Ph_f$ are subclasses of $ID_0$.
\end{enumerate}
\end{prop}

\begin{proof}
Use Proposition 3.18 of \cite{S10} and their
analogue for improper stochastic integrals of $\int_{0+}^c$ type, and note
Proposition \ref{p4} above.
\end{proof}

\section{Domains and ranges of some stochastic integral mappings\\
and their conjugates}

We tackle the problem to find explicit description of the
domains and the ranges of the stochastic integral mappings $\bar\Ph_{p,\al}$,
$\Ld_{q,\al}$, $\Ps_{\al,\bt}$, and their conjugates. 

\medskip
1. {\it $\bar\Ph_{p,\al}$ and its conjugate}.  
Given $p>0$ and $-\infty<\al<2$, let $a=0$, $b=1$, and $h(u)=\Gm(p)^{-1}
(1-u)^{p-1} u^{-\al-1}$. Then $h$ satisfies (C$_1$). We have   
$c=\Gm(|\al|)/\Gm(p+|\al|)$ if $\al<0$, and $c=\infty$ if $\al\ges 0$.
The mapping  $\Ld_h$ is denoted by $\bar \Ph_{p,\al}$, as in \cite{S10}.
It is extensively studied in \cite{S06b} in the 
notation $\Ph_{\bt,\al}=\bar\Ph_{\bt-\al,\al}$, and in \cite{S10}. 
The classes $\mfr R(\Ld_h)$, $\mfr R^{\mrm e}(\Ld_h)$, and
$\mfr R^0(\Ld_h)$ are denoted in \cite{S10} by $K_{p,\al}$,  $K_{p,\al}^{\mrm e}$,  
and $K_{p,\al}^0$, respectively.
We  have, as $s\to\infty$,
\begin{equation}\label{e1.1}
f(s)\sim \begin{cases} \exp[C-\Gm(p) s] &\quad\text{if $\al=0$}\\
(\al\Gm(p) s)^{-1/\al} &\quad\text{if $0<\al<2$} \end{cases}
\end{equation}
with a real constant $C$ depending on $p$. If $\al=1$, then the following
more precise estimate is needed in the analysis of the domain as in
Theorem 4.4 of \cite{S10}:
\begin{equation}\label{e1.0}
f(s)=(\Gm(p)s)^{-1}-(1-p)(\Gm(p)s)^{-2} \log s+O(s^{-2}).
\end{equation}
We have $a_*=1$, $b_*=\infty$, and
\[
h^*(u)=\Gm(p)^{-1} (1-u^{-1})^{p-1} u^{\al+1-4}=\Gm(p)^{-1}(u-1)^{p-1}
u^{\al-p-2},
\]
which satisfies (C$_2$). Thus
\begin{gather*}
g_*(t)=\Gm(p)^{-1} \int_t^{\infty} (u-1)^{p-1} u^{\al-p-2}du, \quad t\in
(1,\infty),\\
c_*=g_*(1+)=\Gm(2-\al)/\Gm(p+2-\al),
\end{gather*}
and $f_*(s)$ for  $s\in (0,c_*)$ is the inverse function of $g_*$.  
For all $\al\in(-\infty,2)$,
\begin{gather}
g_*(t)\sim (2-\al)^{-1} \Gm(p)^{-1} t^{-(2-\al)},\qquad t\to\infty,\notag\\
\label{e1.2}
f_*(s)\sim ((2-\al)\Gm(p) s)^{-1/(2-\al)},\qquad s\dar 0.
\end{gather}
Notice that $\int_0^{c_*} f_*(s)^2 ds<\infty$ if and only if $\al<0$.
If $p=1$, then $f$ and $f_*$ are explicit, namely, $h(u)=u^{-\al-1}$, 
\begin{gather*}
g(t)=\begin{cases} (1/\al) (t^{-\al}-1) &\quad\text{if $\al\neq 0$}\\
-\log t &\quad\text{if $\al = 0$,} \end{cases}
\qquad c=\begin{cases} |\al|^{-1} &\quad\text{if $\al<0$}\\
\infty &\quad\text{if $0\les \al<2$,} \end{cases} \\
f(s)=\begin{cases} (1-|\al|s)^{1/|\al|} &\quad\text{if $\al< 0$}\\
e^{-s} &\quad\text{if $\al = 0$}\\
(1+\al s)^{-1/\al} &\quad\text{if $0<\al<2$,} \end{cases}
\end{gather*}
$h^*(u)=u^{\al-3}$ for $u\in(1,\infty)$ and
\begin{gather*}
g_*(t)=(2-\al)^{-1} t^{-(2-\al)}, \qquad c_*=(2-\al)^{-1},\\
f_*(s)=((2-\al)s)^{-1/(2-\al)} \quad\text{for }s\in (0,(2-\al)^{-1}).
\end{gather*}
The mapping $\bar\Ph_{1,\al}$ was studied by \cite{J88,J89,MMS10,MU10a}.
If $\al=-1$, then $f$ is explicit again:  $g(t)=\Gm(p+1)^{-1} (1-t)^p$,
$c=\Gm(p+1)^{-1}$, and $f(s)=1-(\Gm(p+1)s)^{1/p}$.

In order to describe the ranges, we need two definitions and a proposition. 

\begin{defn}\label{2d1}
(\cite{S10}, p.\,7)  Let $p>0$.  A $[0,\infty]$-valued function $\ph (u)$ on $\R$ 
[resp.\ $\Rpl=(0,\infty)$] 
is said to be {\it monotone of order $p$ 
on $\R$} [resp.\ $\Rpl$] if $\ph (u)$ is locally integrable on $\R$ [resp.\ $\Rpl$] and there is 
a locally finite measure
$\sg$ on $\R$ [resp.\ $\Rpl$] such that
\[
\ph (u)=\Gm(p)^{-1} \int_{(u,\infty)} (r-u)^{p-1} \sg(dr)\quad\text{for $u\in\R$ [resp.\ $\Rpl$]}.
\]
A function $\ph (u)$ on $\R$ [resp.\ $\Rpl$] is said to be {\it completely monotone on $\R$} 
[resp.\ $\Rpl$] if it is monotone of order $p$ on $\R$ [resp.\ $\Rpl$] for all $p>0$.
\end{defn}

\begin{defn}\label{2d2}
Let $p>0$.  A $[0,\infty]$-valued function $\ph (u)$ on $\R$ [resp.\ $\Rpl$] 
is said to be {\it increasing of order $p$ 
on $\R$} [resp.\ $\Rpl$] if $\ph (u)$ is locally integrable on $\R$ [resp.\ $\Rpl$] and there is 
a locally finite measure
$\sg$ on $\R$ [resp.\ $\Rpl$] such that
\begin{gather*}
\ph (u)=\Gm(p)^{-1} \int_{(-\infty,u)} (u-r)^{p-1} \sg(dr)\quad\text{for }u\in\R\\
[\text{resp.\ }\ph (u)=\Gm(p)^{-1} \int_{(0,u)} (u-r)^{p-1} \sg(dr)\quad\text{for }u\in\Rpl].
\end{gather*}
A function $\ph (u)$ on $\R$ [resp.\ $\Rpl$] is said to be {\it completely increasing on $\R$} 
[resp.\ $\Rpl$] if it is increasing of order $p$ on $\R$ [resp.\ $\Rpl$] for all $p>0$.
\end{defn}

\begin{prop}\label{2p1}
Let $p>0$. Let $\ph (u)$ be a $[0,\infty]$-valued function on $\R$ $[$resp.\ $\Rpl]$.  
Then $\ph(u)$ is increasing of order 
$p$ on $\R$ $[$resp.\ $\Rpl]$ if and only if $\ph (-u)$ $[$resp.\ $u^{p-1}\ph (u^{-1})]$ 
is monotone of order $p$ on $\R$ $[$resp.\ $\Rpl]$.
In other words, $\ph(u)$ is monotone of order 
$p$ on $\R$ $[$resp.\ $\Rpl]$ if and only if $\ph (-u)$ $[$resp.\ $u^{p-1}\ph (u^{-1})]$ 
is increasing of order $p$ on $\R$ $[$resp.\ $\Rpl]$. 
\end{prop}

\begin{proof}
If $\ph (u)$ is increasing of order $p$ on $\R$, then $\ph (-u)$  
is monotone of order $p$ on $\R$, and conversely, since
\[
\ph (-u)=\Gm(p)^{-1} \int_{(-\infty,-u)} (-u-r)^{p-1} \sg(dr)=\Gm(p)^{-1} \int_{(u,\infty)}
(r-u)^{p-1}\sg^{-}(dr)
\]
with $\sg^{-}(B)=\sg(-B)$ for $B\in\mcal B(\R)$. 
If $\ph (u)$ is increasing of order $p$ on $\Rpl$, then $u^{p-1}\ph (u^{-1})$  
is monotone of order $p$ on $\Rpl$, and conversely, since
\begin{align*}
u^{p-1}\ph (u^{-1})&=\Gm(p)^{-1}u^{p-1} \int_{(0,u^{-1})} (u^{-1}-r)^{p-1} \sg(dr)\\
&=\Gm(p)^{-1} \int_{(u,\infty)}(r-u)^{p-1} r^{1-p} \sg^{-1}(dr)
\end{align*}
with $\sg^{-1}(B)=\int_{\Rpl} 1_B(r^{-1})\sg(dr)$ for $B\in\mcal B(\Rpl)$.
Note that, on $\Rpl$, $\ps(u)=u^{p-1}\ph (u^{-1})$ if and only if $\ph(u)=u^{p-1}\ps (u^{-1})$.
\end{proof}

\noindent {\it Remark}.  Assume that $p$ is a positive integer and that $\ph(u)$
is $p$ times differentiable.  Then $\ph(u)$ is monotone of order $p$
on $\R$ [resp.\ $\Rpl$] if and only if $(-d/du)^n \ph(u)\ges 0$ for $n=1,2,\ldots,p$ 
on $\R$ [resp.\ $\Rpl$] and $\ph(u)\to 0$ as $u\to\infty$;
$\ph(u)$ is increasing of order $p$
on $\R$ [resp.\ $\Rpl$] if and only if $(d/du)^n \ph\ges 0$ for $n=1,2,\ldots,p$ 
on $\R$ [resp.\ $\Rpl$] and $\ph(u)\to 0$ as $u\to -\infty$ [resp.\ $u\to 0$].  
See Corollary 2.12 of \cite{S10} for the proof of the first assertion.
The proof of the second assertion is given by a modification of that of the first.

A function $\ph (u)$ is completely increasing on $\R$ if and only if $\ph(-u)$ is 
completely monotone on $\R$.  If a function $\ph (u)$ is completely increasing on $\Rpl$,
then $\ph(u^{-1})$ is completely monotone on $\Rpl$.  However, the converse of the last
statement is not true; consider $\ph(u)=u^{\al}$ with $\al$ being positive and non-integer.

\medskip
We will also use the concepts for $\mu\in ID$
to have  weak mean $m_{\mu}$ and to have weak mean $m_{\mu}$ absolutely, 
introduced in \cite{S10}.  We say that $\mu\in ID$ has weak
mean $m_{\mu}$ if $\int_{1<|x|\les a} x\nu_{\mu}(dx)$ is convergent in $\R^d$ 
as $a\to\infty$ and if 
\[
\wh\mu(z)=\exp\Bigl[ -\tfrac12 \langle z,A_{\mu} z\rangle +\lim_{a\to\infty}
\int_{|x|\les a}(e^{i\langle z,x\rangle}-1-i\langle z,x\rangle) \nu_{\mu}(dx)
+i\langle m_{\mu},z\rangle\Bigr],\quad z\in\R^d.
\]
If $\mu\in ID$ has weak mean $m_{\mu}$, then we have 
$m_{\mu}=\gm_{\mu}+\lim_{a\to\infty}\int_{1<|x|\les a} x\nu_{\mu}(dx)$.  
We say that $\mu\in ID$ has weak mean $m_{\mu}$ absolutely if $\mu$ has weak
mean $m_{\mu}$ and if $\int_{(1,\infty)}r\bar\nu_{\mu}(dr)\left|\int_{S}
\xi\ld_r^{\mu}(d\xi)\right|<\infty$, where  $(\bar\nu_{\mu}(dr),\ld_r^{\mu}(d\xi))$
is a spherical decomposition of $\nu_{\mu}$.
This property is independent of the choice of spherical decompositions of $\nu_{\mu}$.
If $\mu\in ID$ has mean $m_{\mu}$, then $\mu$ has weak mean $m_{\mu}$ absolutely.

Now let us give description of the domains and the ranges of $\bar\Ph_{p,\al}$ and 
its conjugate.  The results on $\bar\Ph_{p,\al}$ are already known; our emphasis
lies on the counterpart in the results on their conjugates.

\begin{thm}\label{2t1new}
Let $p>0$ and $-\infty<\al<2$. Let $\Ld_h=\bar\Ph_{p,\al}$.
\begin{enumerate}
\item The domains and the ranges in the essentially definable sense are
as follows:
\begin{align}\label{2t1new.1}
\mfr D^{\mrm{e}} (\Ld_h)&= \begin{cases} ID\quad &\text{if $\al<0$}\\
\{\rh\in ID\cl {\ts\int_{|x|>1}} \log |x|\nu_{\rh}(dx)<\infty\}
\quad &\text{if $\al=0$}\\
\{\rh\in ID\cl {\ts\int_{|x|>1}}|x|^{\al}
\nu_{\rh}(dx)<\infty\} \quad &\text{if $0<\al<2$},
\end{cases}\\
\mfr R^{\mrm e}(\Ld_h)&=\{ \mu\in ID\colon \text{$\nu_{\mu}$ has a 
rad.\ dec.\ $(\ld(d\xi),u^{-\al-1}k_{\xi}(u)du)$ such that $k_{\xi}(u)$}
\nonumber \\
&\text{is measurable in $(\xi,u)$ and monotone of order $p$ 
in $u\in\Rpl$}\} \quad\text{for all $\al$},\label{2t1new.2}\\
\label{2t1new.3} \mfr D^{\mrm{e}} (\Ld_h^*)&= \begin{cases} ID\quad &\text{if $\al<0$}\\
\{\rh\in ID_0\cl {\ts\int_{|x|<1}} (-\log |x|)|x|^2\nu_{\rh}(dx)<\infty\}
\quad &\text{if $\al=0$}\\
\{\rh\in ID_0\cl {\ts\int_{|x|<1}}|x|^{2-\al}
\nu_{\rh}(dx)<\infty\} \quad &\text{if $0<\al<2$},
\end{cases}\\
\mfr R^{\mrm e}(\Ld_h^*)_0&=\{ \mu\in ID_0\colon \text{$\nu_{\mu}$ has a 
rad.\ dec.\ $(\ld(d\xi),u^{\al-p-2}k_{\xi}(u)du)$ such that $k_{\xi}(u)$}
\nonumber \\
\label{2t1new.4} &\text{is measurable in $(\xi,u)$ and increasing of order $p$ in $u\in\Rpl$}\}
\quad\text{for all $\al$},\\
\label{2t1new.4a} \mfr R^{\mrm e}(\Ld_h^*)&= \begin{cases}
\{ \mu_1*\mu_0\cl \mu_1\in\mfr S_2^0,\,\mu_0\in\mfr R^{\mrm e}(\Ld_h^*)_0\}
\quad &\text{if $\al<0$}\\
\mfr R^{\mrm e}(\Ld_h^*)_0 \quad &\text{if $0\les\al<2$}.
\end{cases}
\end{align}
\item If $-\infty<\al<1$, then
\begin{align}
\label{2t1new.5} \mfr D(\Ld_h)&=\mfr D^0(\Ld_h)=\mfr D^{\mrm e}(\Ld_h),\\
\label{2t1new.6} \mfr R(\Ld_h)&=\mfr R^0(\Ld_h)=\mfr R^{\mrm e}(\Ld_h),\\
\label{2t1new.7} \mfr D(\Ld_h^*)&=\mfr D^0(\Ld_h^*)=\mfr D^{\mrm e}(\Ld_h^*),\\
\label{2t1new.8} \mfr R(\Ld_h^*)&=\mfr R^0(\Ld_h^*)=\mfr R^{\mrm e}(\Ld_h^*).
\end{align}
\item If $\al=1$, then
\begin{align}
\label{2t1new.9} \mfr D(\Ld_h)&=\{\rh\in \mfr D^{\mrm e}(\Ld_h)\colon 
m_{\rh}=0,\;\lim_{r\to\infty} {\ts\int_1^r s^{-1}ds\int_{|x|>s}}
x\nu_{\rh}(dx)\text{ exists in $\R^d$}\},\\
\label{2t1new.10} \mfr D^0(\Ld_h)&=\{\rh\in \mfr D^{\mrm e}(\Ld_h)\colon 
m_{\rh}=0,\;{\ts\int_1^{\infty}s^{-1}ds \left|\int_{|x|>s}x\nu_{\rh}(dx)\right|}
<\infty\},\\
\label{2t1new.11} \mfr R(\Ld_h)&=\{\mu\in \mfr R^{\mrm e}(\Ld_h)\colon 
\text{$\mu$ has weak mean $0$}\},\\
\label{2t1new.12} \mfr R^0(\Ld_h)&=\{\mu\in \mfr R^{\mrm e}(\Ld_h)\colon 
\text{$\mu$ has weak mean $0$ absolutely}\},\\
\label{2t1new.13} \mfr D(\Ld_h^*)&=\{\rh\in \mfr D^{\mrm e}(\Ld_h^*)\colon 
\gm_{\rh}^0=0,\;\lim_{r\dar 0} {\ts\int_r^1 s^{-1}ds\int_{|x|<s}}
x\nu_{\rh}(dx)\text{ exists in $\R^d$}\},\\
\label{2t1new.14} \mfr D^0(\Ld_h^*)&=\{\rh\in \mfr D^{\mrm e}(\Ld_h^*)\colon 
\gm_{\rh}^0=0,\;{\ts\int_0^1 s^{-1}ds \left|\int_{|x|<s}x\nu_{\rh}(dx)\right|}
<\infty\}.
\end{align}
\item If\/ $1<\al<2$, then
\begin{align}
\label{2t1new.15} \mfr D(\Ld_h)&=\mfr D^0(\Ld_h) =\{\rh\in \mfr D^{\mrm e}(\Ld_h)\colon 
\text{$\rh$ has mean $0$} \},\\
\label{2t1new.16} \mfr R(\Ld_h)&=\mfr R^0(\Ld_h) =\{\mu\in \mfr R^{\mrm e}(\Ld_h)\colon 
\text{$\mu$ has mean $0$} \},\\
\label{2t1new.17} \mfr D(\Ld_h^*)&=\mfr D^0(\Ld_h^*) =\{\rh\in \mfr D^{\mrm e}(\Ld_h^*)\colon 
\text{$\rh$ has drift $0$} \},\\
\label{2t1new.18} \mfr R(\Ld_h^*)&=\mfr R^0(\Ld_h^*) =\{\mu\in \mfr R^{\mrm e}(\Ld_h^*)\colon 
\text{$\mu$ has drift $0$} \}.
\end{align}
\end{enumerate}
\end{thm} 

\noindent {\it Remark}. 
The expression \eqref{2t1new.4} can be replaced by the following:
\begin{gather}
\mfr R^{\mrm e}(\Ld_h^*)_0=\{ \mu\in ID_0\colon \text{$\nu_{\mu}$ has a 
rad.\ dec.\ $(\ld(d\xi),u^{\al-3}k_{\xi}(u^{-1})du)$ such that $k_{\xi}(v)$}
\nonumber \\
\label{2t1new.19} \text{is measurable in $(\xi,v)$ and monotone of order $p$ in $v\in\Rpl$}\}
\quad\text{for all $\al$}.
\end{gather}

\noindent {\it Remark}.  Description of $\mfr R(\Ld_h^*)$ and
$\mfr R^0(\Ld_h^*)$ in case $\al=1$ for $\Ld_h=\bar\Ph_{p,\al}$
will be given in another paper \cite{SU11}.

\begin{proof}[Proof of Theorem \ref{2t1new}] 
All assertions concerning $\Ld_h$ are known; see
\cite{S10} (Theorems 4.2, 4.15, 4.18, 
and 4.21) and also \cite{S06b} together with \eqref{e1.1} and \eqref{e1.0}.
Then, applying Theorem \ref{t1}, we obtain all results on the domains and
the ranges of $\Ld_h^*$ intersected with $ID_0$.  Thereafter use 
Proposition \ref{p5} to remove the restriction to $ID_0$, recalling that
$\int_0^c f(s)^2ds<\infty$ for all $\al$ and that $\int_0^{c_*} f_* (s)^2ds
<\infty$ if and only if $\al<0$, from \eqref{e1.1} and \eqref{e1.2}.
The details of the proof of \eqref{2t1new.4}, \eqref{2t1new.13}, 
and \eqref{2t1new.14} are as follows.

Let us show \eqref{2t1new.4}. Assume 
$\mu\in\mfr R^{\mrm e}(\Ld_h)_0$. Then $\mu'\in\mfr R^{\mrm e}(\Ld_h^*)_0$
from Theorem \ref{t1} and, since
\[
\nu_{\mu}(B)=\int_S \ld(d\xi) \int_0^{\infty} 1_B(u\xi) u^{-\al-1}k_{\xi}(u)du
\]
with $k_{\xi}(u)$ monotone of order $p$ in $u\in\Rpl$, 
we have, from \eqref{d2a1}, 
\begin{align*}
\nu_{\mu'}(B)&=\int_S \ld(d\xi) \int_0^{\infty} 1_B(u^{-1}\xi) u^{1-\al}k_{\xi}(u)du
=\int_S \ld(d\xi) \int_0^{\infty} 1_B(v\xi) v^{\al-3}k_{\xi}(v^{-1})dv\\
&=\int_S \ld(d\xi) \int_0^{\infty} 1_B(v\xi) v^{\al-p-2} v^{p-1}k_{\xi}(v^{-1})dv.
\end{align*}
Hence, exchanging the roles of $\mu$ and $\mu'$, we see from Proposition \ref{2p1} that 
$\mfr R^{\mrm e}(\Ld_h^*)_0$ is a subclass of the right-hand side of \eqref{2t1new.4}. 
Similarly we can prove that $\mfr R^{\mrm e}(\Ld_h^*)_0$ includes the right-hand side 
of \eqref{2t1new.4}. 

To prove \eqref{2t1new.13} and \eqref{2t1new.14},
notice that $m_{\rh}=-\gm_{\rh'}^0$,
\begin{align*}
&\int_1^r s^{-1}ds\int_{|x|>s}x\nu_{\rh}(dx)=\int_1^r s^{-1}ds\int_{|x|^{-1}>s}
|x|^{-2}x|x|^2\nu_{\rh'}(dx)\\
&\qquad=\int_1^r s^{-1}ds\int_{|x|<s^{-1}}
x\nu_{\rh'}(dx)=\int_{r^{-1}}^1 s^{-1}ds\int_{|x|<s}x\nu_{\rh'}(dx),
\end{align*}
and similarly
\[
\int_1^r s^{-1}ds\left| \int_{|x|>s}x\nu_{\rh}(dx)\right|
=\int_{r^{-1}}^1 s^{-1}ds\left| \int_{|x|<s}x\nu_{\rh'}(dx)\right|,
\]
and apply Theorem \ref{t1}.
\end{proof}

2. {\it $\Ld_{q,\al}$ and its conjugate}.  
Given $q>0$ and $-\infty<\al<2$, let $a=0$, $b=1$, and $h(u)=\Gm(q)^{-1}
(-\log u)^{q-1} u^{-\al-1}$. Then $h$ satisfies (C$_1$). We have 
$c=|\al|^{-q}$ if $\al<0$, and $c=\infty$ if $\al\ges 0$. 
The mapping  $\Ld_h$ is denoted by $\Ld_{q,\al}$, as in \cite{S10}; there
it is extensively studied.  
The classes $\mfr R(\Ld_h)$, $\mfr R^{\mrm e}(\Ld_h)$, and
$\mfr R^0(\Ld_h)$ are denoted by $L_{q,\al}$, $L_{q,\al}^{\mrm e}$,  and
$L_{q,\al}^0$, respectively.
It is known that $\Ld_{q_1+q_2,\al}=\Ld_{q_2,\al} \Ld_{q_1,\al}$ for 
$\al\neq1$.  We have
\begin{gather}\label{e3.2a}
f(s)=\exp[-(\Gm(q+1)s)^{1/q}]\quad\text{for }s\in(0,\infty) \quad
\text{if $\al=0$},\\
\label{e3.2}
f(s)\sim (\al\Gm(q) s)^{-1/\al}(\al^{-1} \log s)^{(q-1)/\al},\quad s\to\infty,
\quad\text{if }\al>0
\end{gather}
(Proposition 6.1 of \cite{S10}). 
We have $a_*=1$, $b_*=\infty$, and
\[
h^*(u)=\Gm(q)^{-1}(\log u)^{q-1}u^{\al-3},
\]
which satisfies (C$_2$). Thus
\begin{gather*}
g_*(t)=\Gm(q)^{-1} \int_{\log t}^{\infty} v^{q-1} e^{-(2-\al)v}dv, \quad t\in
(1,\infty),\\
c_*=g_*(1+)=(2-\al)^{-q},
\end{gather*}
and $f_*(s)$ for  $s\in (0,c_*)$ is the inverse function of $g_*$.  
For all $\al\in(-\infty,2)$ we have
\[
g_*(t)\sim (2-\al)^{-1} \Gm(q)^{-1} t^{-(2-\al)}(\log t)^{q-1},\qquad t\to\infty.
\]
It follows from this that
\begin{equation}\label{e3.3}
f_*(s)\sim (2-\al)^{-q/(2-\al)} (\Gm(q) s)^{-1/(2-\al)} (-\log s)^{(q-1)/(2-\al)},
\qquad s\dar 0.
\end{equation}
The proof is left to the reader. Again notice that $\int_0^{c_*} f_*(s)^2 ds
<\infty$ if and only if $\al<0$.
If $q=1$, then $\Ld_{1,\al}=\bar \Ph_{1,\al}$. 
If $\al=0$, then
$\Ld_{q,0}$ with $q=1,2,\ldots$ coincides with the mapping introduced
by \cite{J83}.

\begin{thm}\label{2t2new}
Let $q>0$ and $-\infty<\al<2$. Let $\Ld_h=\Ld_{q,\al}$.
\begin{enumerate}
\item The domains and the ranges in the essentially definable sense are
as follows:
\begin{align}
\label{2t2new.1} \mfr D^{\mrm{e}} (\Ld_h)&= \begin{cases} ID\quad &\text{if $\al<0$}\\
\{\rh\in ID\cl {\ts\int_{|x|>1}} (\log |x|)^q \nu_{\rh}(dx)<\infty\}
\quad &\text{if $\al=0$}\\
\{\rh\in ID\cl {\ts\int_{|x|>2}}(\log |x|)^{q-1}|x|^{\al}
\nu_{\rh}(dx)<\infty\} \quad &\text{if $0<\al<2$},
\end{cases}\\
\nonumber \mfr R^{\mrm e}(\Ld_h)&=\{ \mu\in ID\colon \text{$\nu_{\mu}$ has a 
rad.\ dec.\ $(\ld(d\xi),u^{-\al-1}h_{\xi}(\log u)du)$ such that $h_{\xi}(y)$}\\
\label{2t2new.2} &\text{is measurable in $(\xi,y)$ and monotone of order $q$ 
in $y\in\R$}\} \quad\text{for all $\al$},\\
\label{2t2new.3} \mfr D^{\mrm{e}} (\Ld_h^*)&= \begin{cases} ID\quad &\text{if $\al<0$}\\
\{\rh\in ID_0\cl {\ts\int_{|x|<1}} (-\log |x|)^q |x|^2\nu_{\rh}(dx)<\infty\}
\quad &\text{if $\al=0$}\\
\{\rh\in ID_0\cl {\ts\int_{|x|<1/2}}(-\log |x|)^{q-1} |x|^{2-\al}
\nu_{\rh}(dx)<\infty\} &\text{if $0<\al<2$},
\end{cases}\\
\nonumber \mfr R^{\mrm e}(\Ld_h^*)_0&=\{ \mu\in ID_0\colon \text{$\nu_{\mu}$ has a 
rad.\ dec.\ $(\ld(d\xi),u^{\al-3}h_{\xi}(\log u)du)$ such that $h_{\xi}(y)$}\\
\label{2t2new.4} &\text{is measurable in $(\xi,y)$ and 
increasing of order $q$ in $y\in\R$}\} \quad\text{for all $\al$},\\
\label{2t2new.4a} \mfr R^{\mrm e}(\Ld_h^*)&= \begin{cases}
\{ \mu_1*\mu_0\cl \mu_1\in\mfr S_2^0,\,\mu_0\in\mfr R^{\mrm e}(\Ld_h^*)_0\}
\quad &\text{if $\al<0$}\\
\mfr R^{\mrm e}(\Ld_h^*)_0 \quad &\text{if $0\les\al<2$}.
\end{cases}
\end{align}
\item If $-\infty<\al<1$, then we have the same assertion as in {\rm (ii)} of
Theorem \ref{2t1new}.
\item If $\al=1$ and $q\ges 1$, then
\begin{align}
\label{2t2new.5} \mfr D(\Ld_h)&=\{\rh\in \mfr D^{\mrm e}(\Ld_h)\colon 
m_{\rh}=0,\;\lim_{r\to\infty} {\ts\int_1^r (\log s)^{q-1}s^{-1}ds\int_{|x|>s}}
x\nu_{\rh}(dx)\text{ exists in $\R^d$}\},\\
\label{2t2new.6} \mfr D^0(\Ld_h)&=\{\rh\in \mfr D^{\mrm e}(\Ld_h)\colon 
m_{\rh}=0,\;{\ts\int_1^{\infty} (\log s)^{q-1}s^{-1}ds \left|\int_{|x|>s}x\nu_{\rh}(dx)\right|}
<\infty\},\\
\label{2t2new.7} \mfr D(\Ld_h^*)&=\{\rh\in \mfr D^{\mrm e}(\Ld_h^*)\colon 
\gm_{\rh}^0=0,\;\lim_{r\dar 0} {\ts\int_r^1 (-\log s)^{q-1}s^{-1}ds\int_{|x|<s}}
x\nu_{\rh}(dx)\text{ exists in $\R^d$}\},\\
\label{2t2new.8} \mfr D^0(\Ld_h^*)&=\{\rh\in \mfr D^{\mrm e}(\Ld_h^*)\colon 
\gm_{\rh}^0=0,\;{\ts\int_0^1 (-\log s)^{q-1}s^{-1}ds \left|\int_{|x|<s}x\nu_{\rh}(dx)\right|}
<\infty\}.
\end{align}
\item If $1<\al<2$, then we have the same assertion as in {\rm (iv)} of
Theorem \ref{2t1new}.
\end{enumerate}
\end{thm} 

\noindent {\it Remark}. In the case $\Ld_h=\Ld_{q,\al}$, we do not know how to describe 
$\mfr D(\Ld_h)_0$, $\mfr D^0(\Ld_h)_0$, $\mfr D(\Ld_h^*)_0$, and 
$\mfr D^0(\Ld_h^*)_0$ for $\al=1$ and $0<q<1$, and $\mfr R(\Ld_h)_0$, 
$\mfr R^0(\Ld_h)_0$, $\mfr R(\Ld_h^*)_0$, and $\mfr R^0(\Ld_h^*)_0$
for $\al=1$ and $q\neq 1$. 
If $\al=1$ and $q= 1$, then $\Ld_{1,1}=\bar\Ph_{1,1}$ and Theorem \ref{2t1new}
applies. 

\begin{proof}[Proof of Theorem \ref{2t2new}] 
This is proved similarly to Theorem \ref{2t1new}.  That is, start from the fact that all 
assertions concerning $\Ld_h$ are known in \cite{S10} (Theorems 6.2, 6.3, 6.9, and 6.12).
Notice that $\int_0^c f(s)^2ds<\infty$ for all 
$\al$ and that $\int_0^{c_*} f_* (s)^2ds
<\infty$ if and only if $\al<0$, from \eqref{e3.2a}--\eqref{e3.3}.
\end{proof}  

3. {\it $\Ps_{\al,\bt}$ and its conjugate}.  
Given $-\infty<\al<2$ and $\bt>0$, let $a=0$, $b=\infty$, and $h(u)=
u^{-\al-1}e^{-u^{\bt}}$. Then $h$ satisfies (C$_1$) and 
$g(t)=\bt^{-1}\int_{t^{\bt}}^{\infty}v^{-\al\bt^{-1}-1}e^{-v}dv$. We have
$c=\bt^{-1} \Gm(|\al|\bt^{-1})$ if $\al<0$, and $c=\infty$ if $\al\ges0$.
The mapping $\Ld_h$ is denoted by $\Ps_{\al,\bt}$ as in 
\cite{MN09,MU10b}. As $t\dar0$,
\begin{equation*}
g(t)=\begin{cases} -\log t+C +o(1)&\quad \text{with some $C\in\R$ if }\al=0\\
\al^{-1} t^{-\al}(1+o(1)) &\quad \text{if }\al>0.
\end{cases}
\end{equation*}
Hence, as $s\to\infty$,
\begin{equation}\label{e5.1}
f(s)\sim \begin{cases} e^{C-s}&\quad \text{if }\al=0\\
(\al s)^{-1/\al} &\quad \text{if }\al>0. \end{cases} 
\end{equation}
If $\al=1$, then more precisely
\begin{equation*}
g(t)= \begin{cases}
t^{-1}+O(1)\quad&\text{if }\bt>1\\
t^{-1}+\log t+O(1)\quad&\text{if }\bt=1\\
t^{-1}-(1-\bt)^{-1}t^{\bt-1} (1+o(1))\quad&\text{if }0<\bt<1.
\end{cases}
\end{equation*}
as $t\dar0$ and it follows that
\begin{equation}\label{e5.2}
f(s)= \begin{cases}
s^{-1}+O(s^{-2})\quad&\text{if }\bt>1\\
s^{-1}-s^{-2}\log s+O(s^{-2})\quad&\text{if }\bt=1\\
s^{-1}+O(s^{-\bt-1})\quad&\text{if }0<\bt<1
\end{cases}
\end{equation}
as $s\to\infty$. We have $a_*=0$, $b_*=\infty$, and
$h^*(u)=u^{\al-3}e^{-u^{-\bt}}$, which satisfies (C$_2$). Thus, 
$g_*(t)=\bt^{-1}\int_0^{t^{-\bt}}v^{(2-\al)\bt^{-1}-1}e^{-v}dv$ 
for $t\in(0,\infty)$ and $c_*=\bt^{-1}\Gm((2-\al)\bt^{-1})$; 
$f_*(s)$ for  $s\in (0,c_*)$ is the inverse function of $g_*$.  
We have, for all $\al$ and $\bt$,
\begin{gather}
g_*(t)\sim (2-\al)^{-1} t^{-(2-\al)},\qquad t\to\infty,\notag \\
\label{e5.4}
f_*(s)\sim ((2-\al)s)^{-1/(2-\al)}, \qquad s\dar0.
\end{gather}
Letting $g_{\al,\bt}$ and $f_{\al,\bt}$ denote the function $g$ and $f$, 
respectively, for given $\al$ and $\bt$, we have
$g_{\al,\bt}(t)=\bt^{-1} g_{\al/\bt,1} (t^{\bt})$ and $f_{\al,\bt}(s)=
f_{\al/\bt,1}(\bt s)^{1/\bt}$.
If $\bt=1$, then $\Ps_{\al,1}$ equals $\Ps_{\al}$ studied in \cite{S06b,S10}. 
In \cite{S10} the classes $\mfr R(\Ps_{\al,1})$, 
$\mfr R^{\mrm e}(\Ps_{\al,1})$, and $\mfr R^0(\Ps_{\al,1})$ are shown to be
equal to $K_{\infty,\al}=\bigcap_{p>0} K_{p,\al}$, 
$K_{\infty,\al}^{\mrm e}=\bigcap_{p>0} K_{p,\al}^{\mrm e}$, and
$K_{\infty,\al}^0=\bigcap_{p>0} K_{p,\al}^0$, respectively.
If $\bt=1$ and $\al=-1$, then $f(s)=-\log s$ for $0<s<1$ and 
$\Ps_{-1,1}$ equals $\Up$ studied in \cite{BMS06,BT04}. 
If $\al=-\bt$, then $f(s)=(-\log \bt s)^{1/\bt}$ for $0<s<1/\bt$ and 
$\Ps_{-\bt,\bt}$ was treated in \cite{ALM10}. If $\al=0$ and $\bt=2$,
$\Ps_{0,2}$ was studied in \cite{AMR08}.

\begin{thm}\label{2t3new}
Let $-\infty<\al<2$ and $\bt>0$. Let $\Ld_h=\Ps_{\al,\bt}$.
\begin{enumerate}
\item The domains and the ranges in the essentially definable sense are
as follows:
\begin{align}\label{2t3new.1}
\mfr D^{\mrm{e}} (\Ld_h)&= \begin{cases} ID\quad &\text{if $\al<0$}\\
\{\rh\in ID\cl {\ts\int_{|x|>1}} \log |x|\nu_{\rh}(dx)<\infty\}
\quad &\text{if $\al=0$}\\
\{\rh\in ID\cl {\ts\int_{|x|>1}}|x|^{\al}
\nu_{\rh}(dx)<\infty\} \quad &\text{if $0<\al<2$},
\end{cases}\\
\mfr R^{\mrm e}(\Ld_h)&=\{ \mu\in ID\colon \text{$\nu_{\mu}$ has a 
rad.\ dec.\ $(\ld(d\xi),u^{-\al-1}k_{\xi}(u^{\bt})du)$ such that $k_{\xi}(v)$}
\nonumber \\
\label{2t3new.2} &\text{is measurable in $(\xi,v)$ and completely monotone 
in $v\in\Rpl$}\} \quad\text{for all $\al$},\\
\label{2t3new.3} \mfr D^{\mrm{e}} (\Ld_h^*)&= \begin{cases} ID\quad &\text{if $\al<0$}\\
\{\rh\in ID_0\cl {\ts\int_{|x|<1}} (-\log |x|)|x|^2\nu_{\rh}(dx)<\infty\}
\quad &\text{if $\al=0$}\\
\{\rh\in ID_0\cl {\ts\int_{|x|<1}}|x|^{2-\al}
\nu_{\rh}(dx)<\infty\} \quad &\text{if $0<\al<2$},
\end{cases}\\
\mfr R^{\mrm e}(\Ld_h^*)_0&=\{ \mu\in ID_0\colon \text{$\nu_{\mu}$ has a 
rad.\ dec.\ $(\ld(d\xi),u^{\al-3}k_{\xi}(u^{-\bt})du)$ such that $k_{\xi}(v)$}
\nonumber \\
\label{2t3new.4} &\text{is measurable in $(\xi,v)$ and completely monotone in $v\in\Rpl$}\}
\quad\text{for all $\al$},\\
\label{2t3new.4a} \mfr R^{\mrm e}(\Ld_h^*)&= \begin{cases}
\{ \mu_1*\mu_0\cl \mu_1\in\mfr S_2^0,\,\mu_0\in\mfr R^{\mrm e}(\Ld_h^*)_0\}
\quad &\text{if $\al<0$}\\
\mfr R^{\mrm e}(\Ld_h^*)_0 \quad &\text{if $0\les\al<2$}.
\end{cases}
\end{align}
\item Concerning $\mfr D$, $\mfr D^0$, $\mfr R$, and $\mfr R^0$ of $\Ld_h$ and $\Ld_h^*$,
the statements in {\rm (ii)}, {\rm (iii)}, and {\rm (iv)} of Theorem \ref{2t1new} are
true word by word.
\end{enumerate}
\end{thm}

\noindent {\it Remark}.  Description of $\mfr R(\Ld_h^*)$ and
$\mfr R^0(\Ld_h^*)$ in case $\al=1$ for $\Ld_h=\Ps_{\al,\bt}$
will be given in another paper \cite{SU11}.

\begin{proof}[Proof of Theorem \ref{2t3new}] 
Concerning the domains of $\Ld_h$ and $\Ld_h^*$, the proof is the same
as Theorem \ref{2t1new}.  The ranges of $\Ld_h$ are given in \cite{S06b} and 
\cite{S10} (Theorems 5.8 and 5.10) in case $\bt=1$ and treated in 
Theorem 2.8 of \cite{MN09} for $\al<1$ and $\bt\neq 1$.
To deal with $\mfr R(\Ld_h)$ and $\mfr R^0(\Ld_h)$ in case $1\les\al<2$
with $\bt\neq 1$, we can extend the method of the proof of Theorem 5.10
of \cite{S10}.  The assertions on the ranges of $\Ld_h^*$ are proved from
the assertions on $\Ld_h$ in the same way as in the proof of 
Theorem \ref{2t1new}.
\end{proof}

\medskip
Maejima and Ueda \cite{MU10a} introduced the class $L^{\la\al\ra}$ and showed 
its connection to $\mfr R^{\mrm e}(\Ld_h)$ with $\Ld_h=\Ld_{1,\al}=\bar\Ph_{1,\al}$.
In the rest of this section we introduce a class $L^{\la\al\ra *}$ and study
its connection to $\mfr R^{\mrm e}(\Ld_h^*)$ with $\Ld_h=\Ld_{1,\al}=\bar\Ph_{1,\al}$.
The definition of $L^{\la\al\ra}$ in \cite{MU10a} is as follows. 
A distribution $\mu\in ID$ is called 
$\al$-selfdecomposable with $\al\in\R$ if, for any $b>1$, there is $\rh_b\in ID$ satisfying
\begin{equation}\label{MU1}
\wh\mu(z)=\wh\mu(b^{-1}z)^{b^{\al}}\wh\rh_b(z),\qquad z\in\R^d.
\end{equation}
The totality of $\al$-selfdecomposable distributions on $\R^d$ is 
denoted by $L^{\la\al\ra}$.  The $0$-selfdecomposability
coincides with the selfdecomposability.  The paper \cite{MU10a} treated  
distributions in $L^{\la\al\ra}$ systematically. Earlier this class
was studied by Jurek \cite{J88,J89} and others in 1970s and 80s 
and also in \cite{MMS10}; see
references in \cite{MU10a}.  The following properties of $L^{\la\al\ra}$ are known.
\begin{enumerate}
\item $\mu\in L^{\la\al\ra}$ if and only if, for any $b>1$, $A_{\mu}-b^{\al-2}
A_{\mu}$ is nonnegative-definite and $\nu_{\mu}\ges b^{\al} T_{b^{-1}}\nu_{\mu}$.
\item If $\mu_1,\mu_2\in L^{\la\al\ra}$, then $\mu_1 *\mu_2 \in L^{\la\al\ra}$.
\item If $\mu\in L^{\la\al\ra}$, then $\mu^s \in L^{\la\al\ra}$ for $s>0$.
\item If $\al_1<\al_2$, then $L^{\la\al_1 \ra}\supset L^{\la\al_2 \ra}$.
\item $\bigcap_{\bt<\al} L^{\la\bt \ra}= L^{\la\al \ra}$.
\item If $\al>2$, then $L^{\la\al\ra}=\{\dl_c\cl c\in\R^d\}$.
\item $L^{\la 2\ra}=\mfr S_2$, the class of $2$-stable (that is, Gaussian) distributions.
\item If $0<\al\les\bt<2$, then $L^{\la\al\ra}$ contains all $\bt$-stable 
distributions.
\item If $0<\bt<\al<2$, then $L^{\la\al\ra}$ does not contain any non-trivial 
$\bt$-stable distribution.
\item Let $-\infty<\al<2$.  Then $\mu\in L^{\la\al\ra}$ if and only if 
$\nu_{\mu}$ has a radial decomposition $(\ld(d\xi),u^{-\al-1}k_{\xi}(u)du)$
such that $k_{\xi}(u)$ is measurable in $(\xi,u)$ and, for $\ld$-a.\,e.\ $\xi$,
decreasing on $\Rpl$ in $u$.
\item If $\al\les0$, then $L^{\la\al\ra}=\mfr R^{\mrm e}(\Ld_{1,\al})$.
\item If $0<\al<2$, then $\mu\in L^{\la\al\ra}$ if and only if 
$\mu=\mu_0 *\mu_1$
where $\mu_0\in\mfr R^{\mrm e}(\Ld_{1,\al})$ and 
$\mu_1$ is $\al$-stable.
\end{enumerate}

\begin{defn}\label{2d3}
Let $\al\in\R$.  The class $L^{\la\al\ra *}$ is the totality of $\mu\in ID$
such that, for any $b>1$, there is $\sg_b\in ID$ satisfying
\begin{equation}\label{2d3-1}
\wh\mu(z)=\wh\mu(bz)^{b^{\al-2}} \wh\sg_b(z),\qquad z\in\R^d.
\end{equation}
\end{defn}

\begin{prop}\label{2p2}
Let $\al\in\R$.  The class $L^{\la\al\ra *}$ has the following properties.
\begin{enumerate}
\item $\mu\in L^{\la\al\ra *}$ if and only if, for any $b>1$, $A_{\mu}-b^{\al}
A_{\mu}$ is nonnegative-definite and $\nu_{\mu}\ges b^{\al-2} T_b\nu_{\mu}$.
\item If $\mu_1,\mu_2\in L^{\la\al\ra *}$, then $\mu_1 *\mu_2 \in L^{\la\al\ra *}$.
\item If $\mu\in L^{\la\al\ra *}$, then $\mu^s \in L^{\la\al\ra *}$ for $s>0$.
\item If $\al_1<\al_2$, then $L^{\la\al_1\ra *}\supset L^{\la\al_2\ra *}$.
\item $\bigcap_{\bt<\al} L^{\la\bt\ra *}= L^{\la\al\ra *}$.
\item If $\al\ges 2$, then $L^{\la\al\ra *}=\{\dl_c\cl c\in\R^d\}$.
\item If $\al>0$, then $L^{\la\al\ra *} \subset ID_0$.
\item $L^{\la 0\ra *}\supset \mfr S_2$.
\item If\/ $0<\al\les\bt<2$, then $L^{\la\al\ra *}$ contains all $(2-\bt)$-stable 
distributions.
\item If\/ $0<\bt<\al<2$, then $L^{\la\al\ra *}$ does not contain any non-trivial 
$(2-\bt)$-stable distribution.
\end{enumerate}
\end{prop}
 
Proof is straightforward.

\begin{thm}\label{2t6}
Let $\al\in\R$ and let $\mu\in ID_0$.  Then $\mu\in L^{\la\al\ra}$ if and only
if $\mu'\in L^{\la\al\ra *}$.
\end{thm}

\begin{proof}
Assume that $\mu\in L^{\la\al\ra}$.  Then $\mu=(T_{b^{-1}} \mu)^{b^{\al}} *
\rh_b$.  Using Proposition \ref{p3}, we obtain $\mu'=((T_{b^{-1}}\mu)')^{b^{\al}}
*\rh'_b= (T_b(\mu'))^{b^{\al-2}}*\rh'_b$.  Hence $\mu'\in L^{\la\al\ra *}$
with $\sg_b=\rh'_b$.  In a similar way we can show that $\mu'\in L^{\la\al\ra *}$
implies $\mu\in L^{\la\al\ra}$. 
\end{proof}

Notice that, for $\al=0$, Theorem \ref{2t6} gives a characterization of the inversions
of selfdecomposable distributions in $ID_0$ in terms of a new kind of decomposability.

\begin{prop}\label{2p3}
Let $-\infty<\al<2$ and $\mu\in ID_0$.  Then  $\mu\in L^{\la\al\ra *}$ if and only
if $\nu_{\mu}$ has a radial decomposition $(\ld(d\xi),u^{\al-3}k_{\xi}(u)du)$
such that $k_{\xi}(u)$ is measurable in $(\xi,u)$ and, for $\ld$-a.\,e.\ $\xi$,
increasing in $u\in\Rpl$.
\end{prop}

\begin{proof}
Using Theorem \ref{2t6} and property (x) of $L^{\la\al\ra}$, we can prove the 
assertion similarly to the proof of \eqref{2t1new.4}.
\end{proof}

\begin{prop}\label{2p4}
If $\al< 0$, then $L^{\la\al\ra *}=\mfr R^{\mrm e}(\Ld_{1,\al}^*)$.
If $0\les \al<2$, then $\mu\in L^{\la\al\ra *}$ if and only if $\mu=\mu_0 *\mu_1$
where $\mu_0\in\mfr R^{\mrm e}(\Ld_{1,\al}^*)=\mfr R^{\mrm e}(\Ld_{1,\al}^*)_0$ 
and $\mu_1$ is $(2-\al)$-stable.
\end{prop}

\begin{proof}
Combine properties (xi) and (xii) of $L^{\la\al\ra}$ with (iii) of Theorem \ref{t0}, 
\eqref{t1.4} of Theorem \ref{t1},  (i) of Theorem \ref{2t2new}, 
and Theorem \ref{2t6}.
\end{proof}


\section{Similar results in other cases}

The definition of the conjugates of stochastic integral mappings $\Ph_f$
is restricted to the case where $f$ is a positive function satisfying some
condition.  We do not know how to define conjugates of general
stochastic integral mappings.  But we can obtain similar results in a case
where $f$ takes positive and negative values both.

\begin{defn}\label{Sd1}
A function $h$ is said to satisfy \emph{Condition} (D) 
if $h$ is defined on $\R\setminus\{0\}$, positive, and measurable, and 
\begin{equation}\label{Sd1.1}
\int_{\R\setminus\{0\}} h(u)(1+u^2)du<\infty.
\end{equation}
For any function $h$ satisfying Condition (D), define 
$h^*(u)=h(1/u)/u^4$ for $u\in\R\setminus\{0\}$.
\end{defn}

\begin{prop}\label{Sp1}
If $h$ satisfies Condition $(\mrm{D})$, then $h^*$ satisfies
Condition $(\mrm{D})$ and $(h^*)^* =h$.
\end{prop}

Proof is straightforward.

\medskip
Let $h$ be a function satisfying Condition (D).  Let 
\begin{equation}\label{S.1}
g_h(t)=\int_{(t,\infty)\setminus \{0\}} h(u)du,\qquad t\in\R.
\end{equation}
Then $g_h(t)$ is strictly decreasing continuous function with
$g_h(-\infty)<\infty$ and $g_h(\infty)=0$.  Let $c_h=g_h(-\infty)$.
Hence $c_h<\infty$.  Define $f_h(s)$ as
\[
s=g_h(t)\text{ with }-\infty<t< \infty\quad\Leftrightarrow\quad 
t=f_h(s)\text{ with }0< s<c_h.
\]
Then $f_h(s)$ is a strictly decreasing, continuous function with
$f_h(0+)=\infty$ and $f_h(c_h-)=-\infty$, and $\int_0^{c_h} f_h(s)^2
ds=\int_{\R\setminus\{0\}} u^2h(u)du<\infty$.  
The improper stochastic integral 
$\lim_{p\dar 0, q\uar c_h} \int_p^q f_h(s)dX_s^{(\rh)}$ is convergent
in probability for all $\rh\in ID$, as is proved in Theorem 6.1 of \cite{S07}.
The distribution of this improper stochastic integral is denoted by $\Ld_h \rh$.  
We have $\mfr D(\Ld_h)=ID$; there is no need to consider essentially 
definable case.  We have $\int_0^{c_h} |\log\wh\rh(f_h(s)z)|ds
<\infty$ for all $\rh\in ID$, that is, $\Ld_h \rh$ is absolutely definable
for all $\rh\in ID$ (see Theorem 6.1 of \cite{S07}). 
It is easy to see that $\Ld_h\rh\in ID_0$ if and only if $\rh\in ID_0$.

\begin{defn}\label{Sd2}
If $h$ satisfies Condition (D), then $\Ld_{h^*}$ is called the \emph{conjugate} of 
$\Ld_h$.  Write $\Ld^*_h=\Ld_{h^*}$.
\end{defn}

It follows from Proposition \ref{Sp1} that the conjugate of $\Ld^*_h$ is $\Ld_h$.

\begin{thm}\label{St1}
Let $h$ be a function satisfying Condition {\rm (D)}. Let $\rh$ and $\mu$ be in $ID_0$.
Then $\Ld_h \rh=\mu$ if and only if $\Ld^*_h \rh'=\mu'$.  Consequently, 
$\mfr R(\Ld^*_h)_0=(\mfr R(\Ld_h)_0)'$.
\end{thm}

\begin{proof}
We proceed as in the proof of Theorem \ref{t1}.  But this time we have to be careful
as $u$ and $1/u$ are discontinuous at $0$ and $f_h(s)$ and $f_{h^*}(s)$ take positive 
and negative values.  Details are omitted.
\end{proof}

\begin{ex}\label{Se1}
Let $h(u)=(2\pi)^{-1/2} e^{-u^2/2}$ for $u\in\R\setminus\{0\}$.  Then $h$ satisfies
Condition (D) and $h^*(u)=(2\pi)^{-1/2} e^{-1/(2u^2)} (1/u^4)$ for $u\in\R\setminus\{0\}$
and $c_h=1$.  Let $ID^{\mrm{sym}}$ denote the class of symmetric  infinitely divisible
distributions on $\R^d$.  Then
\begin{gather*}
\mfr R(\Ld_h)=\{ \mu\in ID^{\mrm{sym}} \colon \text{$\nu_{\mu}$ has a rad.\ dec.\ 
$(\ld(d\xi),k_{\xi}(u^2)du)$ such that $k_{\xi}(v)$ is}\\
\text{measurable in $(\xi,v)$ and, for $\ld$-a.\,e.\ $\xi$,
completely monotone in $v\in\Rpl$}\}.
\end{gather*}
This is essentially a result of \cite{AM07,MR02}.  The class $\mfr R(\Ld_h)$ is 
identical with the class of type $G$ distributions on $\R^d$ of Maejima and
Rosi\'nski \cite{MR02}.
Then, Theorem \ref{St1} and a discussion similar to the proof of Theorem \ref{2t3new} 
show that
\begin{gather*}
\mfr R(\Ld^*_h)_0 =(\mfr R(\Ld_h)_0)'=\{ \mu\in ID^{\mrm{sym}}_0 \colon \text{$\nu_{\mu}$ 
has a rad.\ dec.\ $(\ld(d\xi),u^{-4}k_{\xi}(u^{-2})du)$ such that}\\
\text{$k_{\xi}(v)$ is measurable in $(\xi,v)$ and, for $\ld$-a.\,e.\ $\xi$, 
completely monotone in $v\in\Rpl$}\}.
\end{gather*}
For $\mu\in ID$ let $\mu_0$ and $\mu_1$ denote the  infinitely divisible distributions
with triplets $(0,\nu_{\mu},\gm_{\mu})$ and $(A_{\mu},0,0)$, respectively. Then we
have
\[
\mfr R(\Ld^*_h)=\{\mu\in ID^{\mrm{sym}} \colon \mu=\mu_1*\mu_0\text{ with }
\mu_1\in\mfr S_2^0,\,\mu_0\in\mfr R(\Ld_h^*)_0\}
\]
similarly to Proposition \ref{p5}, since $0<\int_0^{c_{h^*}} f_{h^*}(s)^2 ds
<\infty$.
\end{ex}

\section{Limits of some nested classes}

Let us make a study of the limit of the ranges of the iterations of $\Ld_h$ and
$\Ld_h^*$. The iteration $\Ph_f^n$ of a stochastic integral mapping
$\Ph_f$ is defined as $\Ph_f^1=\Ph_f$ and $\Ph_f^{n+1} \rh=\Ph_f(\Ph_f^n \rh)$ with
$\mfr D(\Ph_f^{n+1})=\{\rh\in \mfr D(\Ph_f^n) \cl \Ph_f^n \rh\in \mfr D(\Ph_f)\}$.  
We have
$ID\supset \mfr R(\Ph_f)\supset \mfr R(\Ph_f^2)\supset \cdots$. The limit class is denoted
by $\mfr R_{\infty}(\Ph_f)=\bigcap_{n=1}^{\infty} \mfr R(\Ph_f^n)$.
In the case where $\Ld_h$ equals $\bar\Ph_{p,\al}$, $\Ld_{q,\al}$, 
or $\Ps_{\al,\bt}$, the description of $\mfr R_{\infty}(\Ld_h)$ is studied in 
\cite{ALM10,MS09,MU10b,MU11,RS03,S10,S11}. In \cite{S11} it is obtained
for $\al\in(-\infty,1)\cup(1,2)$, $p\ges 1$, $q>0$, and $\bt=1$ and for $\al=1$, 
$p\ges 1$, $q=1$, and $\bt=1$; now we want
to describe $\mfr R_{\infty}(\Ld_h^*)$.  We will see new classes appear as 
$\mfr R_{\infty}(\Ld_h^*)$ for some parameter values. We are also interested in 
finding what parameters are relevant.  It will be shown that only the parameter
$\al$ is relevant and the parameters $p$ and $q$ are irrelevant. 
 
We need the class  $L_{\infty}$ in the study of $\mfr R_{\infty}(\Ph_f)$.
It is the class of completely selfdecomposable
distributions on $\R^d$, which is the smallest class that is closed under convolution
and weak convergence and contains all stable distributions on $\R^d$. 
A distribution $\mu
\in ID$ belongs to $L_{\infty}$ if and only if $\nu_{\mu}$ is
represented as 
\begin{equation}\label{5-1}
\nu_{\mu}(B)=\int_{(0,2)} \Gm_{\mu}(d\bt)\int_S \ld_{\bt}^{\mu}(d\xi) \int_0^{\infty}
1_B(r\xi) r^{-\bt-1} dr,\quad B\in\mcal B(\R^d),
\end{equation}
where $\Gm_{\mu}$ is a measure on the open interval
$(0,2)$ satisfying $\int_{(0,2)} (\bt^{-1}+(2-\bt)^{-1}) \Gm_{\mu}(d\bt)\break <\infty$
and $\{\ld_{\bt}^{\mu}\cl \bt\in(0,2)\}$ is a measurable family of probability
measures on $S$.  This $\Gm_{\mu}$ is uniquely
determined by $\nu_{\mu}$ and  $\{\ld_{\bt}^{\mu}\}$ is determined by $\nu_{\mu}$
up to $\bt$ of $\Gm_{\mu}$-measure $0$.  
For a Borel subset $E$ of the interval $(0,2)$, the class $L_{\infty}^E$ denotes
the totality of $\mu\in L_{\infty}$ such that $\Gm_{\mu}$ is concentrated on $E$.
The class $L_{\infty}^{(\al,2)}$ for $0<\al<2$ appears in \cite{MU11,S10,S11} in the
description of $\mfr R_{\infty}(\Ph_f)$ for some $f$.
We will use $(L_{\infty})_0=L_{\infty}\cap ID_0$ and $(L_{\infty}^E)_0=L_{\infty}^E
\cap ID_0$.

\begin{prop}\label{5p1}
Let $E\in\mcal B((0,2))$.  Let\/ $2-E$ denote the set $\{2-\bt\colon \bt\in E\}$.  Let 
$\mu\in (L_{\infty})_0$.  Then $\mu\in L_{\infty}^E$ if and 
only if $\mu'\in L_{\infty}^{2-E}$.
\end{prop}

\begin{proof}
Assume that $\mu\in L_{\infty}^E$.  Then $\Gm_{\mu}((0,2)\setminus E)=0$.
Define $\Gm^{\sharp}(F)=\int_{(0,2)} 1_F(2-\bt) \Gm_{\mu}(d\bt)$
for $F\in\mcal B((0,2))$.  Then $\int_{(0,2)}(\bt^{-1}
+(2-\bt)^{-1})\Gm^{\sharp}(d\bt)<\infty$ and $\Gm^{\sharp}((0,2)\setminus (2-E))=0$.
For $B\in\mcal B(\R^d)$ we have, from \eqref{d2a1} and \eqref{5-1},
\begin{align*}
\nu_{\mu'}(B)&=\int_E \Gm_{\mu}(d\bt)\int_S \ld_{\bt}^{\mu}(d\xi) \int_0^{\infty}
1_B(r^{-1}\xi) r^{-\bt+1} dr\\
&=\int_E \Gm_{\mu}(d\bt)\int_S \ld_{\bt}^{\mu}(d\xi) \int_0^{\infty}
1_B(u\xi) u^{\bt-3} du\\
&=\int_{(0,2)} 1_{2-E}(\bt)\Gm^{\sharp}(d\bt)\int_S \ld_{2-\bt}^{\mu}(d\xi) \int_0^{\infty}
1_B(u\xi) u^{-1-\bt} du.
\end{align*}
Therefore $\mu'\in L_{\infty}^{2-E}$.  Now it is automatic that $\mu'\in L_{\infty}^{2-E}$
implies $\mu\in L_{\infty}^{E}$, since $\mu''=\mu$.
\end{proof}

\begin{prop}\label{5p2}
Let $h$ be a function satisfying Condition {\rm (C)}.  
Let $n$ be a positive integer.  Let $\rh\in ID_0$.  Then
\[
\rh\in \mfr D(\Ld_h^n) \quad\text{and}\quad \Ld_h^n\rh=\mu
\]
if and only if
\[
\rh'\in \mfr D((\Ld_h^*)^n) \quad\text{and}\quad (\Ld_h^*)^n\rh'=\mu' .
\]
Thus, $\mfr D ((\Ld_h^*)^n)_0= (\mfr D (\Ld_h^n)_0)'$ and\/
$\mfr R ((\Ld_h^*)^n)_0= (\mfr R (\Ld_h^n)_0)'$.
\end{prop}

\begin{proof}
Using Theorem \ref{t1}, we can show the assertion by induction in $n$.
\end{proof}

\begin{thm}\label{5t1}
Let $h$ be a function satisfying Condition {\rm (C)}.  Consider $\Ld_h$ and
$\Ld_h^*$.  Then
\[
\mfr R_{\infty}(\Ld_h^*)_0=(\mfr R_{\infty}(\Ld_h)_0)'.
\]
\end{thm}

\begin{proof}
It follows from $\mfr R ((\Ld_h^*)^n)_0= (\mfr R (\Ld_h^n)_0)'$ that
\[
\mfr R_{\infty}(\Ld_h^*)_0=\bigcap_{n=1}^{\infty}\mfr R((\Ld_h^*)^n)_0=
\bigcap_{n=1}^{\infty} (\mfr R(\Ld_h^n)_0)'=
\left( \bigcap_{n=1}^{\infty}\mfr R(\Ld_h^n)_0 \right)'
=(\mfr R_{\infty}(\Ld_h)_0)',
\]
completing the proof.
\end{proof}

\begin{thm}\label{5t2}
Let $\Ld_h$ be one of\/ $\bar\Ph_{p,\al}$, $\Ld_{q,\al}$, and\/ $\Ps_{\al,1}$.
The classes $\mfr R_{\infty}(\Ld_h)$ and $\mfr R_{\infty}(\Ld_h^*)$ are as
follows.
\begin{enumerate}
\item If $\al<0$, $p\ges1$, and $q>0$, then $\mfr R_{\infty}(\Ld_h)=
\mfr R_{\infty}(\Ld_h^*)=L_{\infty}$.
\item If\/ $0\les \al<1$, $p\ges1$, and $q>0$, then 
\[
\mfr R_{\infty}(\Ld_h)=L_{\infty}^{(\al,2)},\qquad \mfr R_{\infty}(\Ld_h^*)
=(L_{\infty}^{(0,2-\al)})_0.
\]
\item If\/ $\al=1$, $p\ges1$, and $q=1$, then 
\[
\mfr R_{\infty}(\Ld_h)=L_{\infty}^{(1,2)}\cap\{\mu\in ID\cl \text{$\mu$ has 
weak mean $0$}\}.
\]
\item If\/ $1<\al<2$, $p\ges1$, and $q>0$, then 
\begin{align*}
\mfr R_{\infty}(\Ld_h)&=L_{\infty}^{(\al,2)}\cap\{\mu\in ID\cl \text{$\mu$ has 
mean $0$}\},\\
\mfr R_{\infty}(\Ld_h^*)
&=(L_{\infty}^{(0,2-\al)})_0 \cap\{\mu\in ID_0\cl \text{$\mu$ has drift $0$}\}.
\end{align*}
\end{enumerate}
\end{thm}

\begin{proof}
All results on $\mfr R_{\infty}(\Ld_h)$ are given in \cite{S11}.
Hence, using Theorem \ref{5t1},
we see from Propositions \ref{5p1} and \ref{p1a} (v) that $\mfr R_{\infty}
(\Ld_h^*)_0$
equals $(L_{\infty})_0$, $(L_{\infty}^{(0,2-\al)})_0$, or $(L_{\infty}^{(0,2-\al)})_0 \cap
\{\mu\in ID_0\cl \text{$\mu$ has drift $0$}\}$ in (i), (ii), or (iv), respectively.

Let us prove (i).  Note that $0<\int_0^{c_{h^*}} f_{h^*}(s)ds<\infty$. 
In general, we can prove $\mfr R(\Ph_f^n)=\{\mu_1*\mu_0\cl \mu_1\in\mfr S_2^0,\,
\mu_0\in\mfr R(\Ph_f^n)_0\}$ and $\mfr R_{\infty}(\Ph_f)=\{\mu_1*\mu_0\cl 
\mu_1\in\mfr S_2^0,\,\mu_0\in\mfr R_{\infty}(\Ph_f)_0\}$ in (i) of 
Proposition \ref{p5}, repeating the same argument.  Hence 
$\mfr R_{\infty}(\Ld_h^*)=L_{\infty}$.

Proof of (ii) and (iv) is as follows.  In this case we have
$\int_0^{c_{h^*}} f_{h^*}(s)^2 ds=\infty$.
Hence  $\mfr R_{\infty}(\Ld_h^*)\subset ID_0$, as in (ii) of
Proposition \ref{p5}.
\end{proof}

\noindent {\it Remark}.  A supplement of Theorem \ref{5t2} (iii) for
$\mfr R_{\infty}(\Ld_h^*)$ will be given in \cite{SU11}.
Theorem  \ref{5t2} does not cover $\Ps_{\al,\bt}$ with $\bt\neq 1$,
because we rely on the results of \cite{S11}.
However, using the result of \cite{MU10b}, we can extend 
Theorem \ref{5t2} to some of $\Ps_{\al,\bt}$ with $\bt\neq 1$.

\medskip
\noindent {\it Acknowledgments}.   The author thanks Makoto Maejima, Yohei
Ueda, and an anonymous referee for their valuable advice on improvement of the paper.

\end{document}